\documentclass[11pt, reqno]{amsart}
\usepackage{amsmath,amsfonts,amssymb,amsthm,enumerate,appendix,caption}
\usepackage{cmap,mathtools}
\usepackage{paralist,bm}
\usepackage{graphics} 
\usepackage{epsfig} 
\usepackage{graphicx}\usepackage{epstopdf}
 \usepackage[colorlinks=true]{hyperref}
\usepackage{xcolor}
\definecolor{myblue}{rgb}{0.0, 0.0, 1.0}
\definecolor{mygreen}{rgb}{0.01,0.75,0.20}
\hypersetup{ linkcolor=myblue, colorlinks=true, urlcolor=black, citecolor=red}
\usepackage{hyperref}
\usepackage[hyperpageref]{backref}

\newtheorem{theorem}{Theorem}[section]
\newtheorem{corollary}{Corollary}

\newtheorem{lemma}[theorem]{Lemma}
\newtheorem{proposition}{Proposition}

\theoremstyle{definition}
\newtheorem{definition}[theorem]{Definition}
\newtheorem{remark}{Remark}

\newcommand{\ep}{\varepsilon}

\usepackage[margin=1in]{geometry}

\DeclarePairedDelimiter\abs{\lvert}{\rvert}%
\DeclarePairedDelimiter\norm{\lVert}{\rVert}%
\makeatletter
\let\oldnorm\norm
\def\norm{\@ifstar{\oldnorm}{\oldnorm*}}

\newcommand{\al} {\alpha}

\newcommand{\pa} {\partial}

\newcommand{\de} {\delta}
\newcommand{\De} {\Delta}
\newcommand{\Dep} {\Delta_p}
\newcommand{\ga} {\gamma}

\newcommand{\Om} {\Omega}

\newcommand{\la} {\lambda}
\newcommand{\La} {\Lambda}

\newcommand{\Gr} {\nabla}

\newcommand{\no} {\nonumber}
\newcommand{\noi} {\noindent}

\newcommand{\ra} {\rightarrow}

\newcommand{\wra} {\rightharpoonup}
\newcommand{\cd}{\circledast}

\newcommand\restr[2]{{
  \left.\kern-\nulldelimiterspace 
  #1 
  \right|_{#2} 
  }}
  \newcommand{\Halu}{H_{\widetilde{\alpha}}}
  \newcommand{\Omu}{\Om_{H_{\widetilde{\alpha}}}}

\def\w{{\widetilde w}}

\def \uphi {{\underline{\phi}}}
\def \ophi {{\overline{\phi}}}
\def \ug {{\underline{g}}}
\def \uv {{\underline{V}}}
\def \og {{\overline{g}}}
\def \ov {{\overline{V}}}

\def\wp{{W^{1,p}_0(\Om)}}
\def\w2{{W^{1,2}_0(\Om)}}
\def\hh2{{H^1_0(\Om)}}

\def\C{{\mathcal C}}

\def\E{{\mathcal E}}

\def\N{{\mathbb N}}
\def\F{{\mathcal F}}

\def\S{\mathbb{S}}

\def\R{{\mathbb R}}
\def\RN{{\mathbb R}^N}

\def\({{\Big(}}
\def\){{\Big)}}

\def\ws2{{\F_{\frac{N}{2}}}}

\def\c1{{\C_c^1}}
\def\dis{{\displaystyle \int_{\Omega}}}

\def\p{{p^{\prime}}}

\def\f{{\tilde{f}}}
\def\g{{\tilde{g}}}

\def\dx{{\rm d}x}

\def\H{{\mathcal{H}}}

\usepackage[foot]{amsaddr}
\usepackage[pagewise]{lineno}\linenumbers
\nolinenumbers
\makeatother

\title{On the Optimization of the first weighted eigenvalue}

\author[Nirjan Biswas, Ujjal Das, and Mrityunjoy Ghosh]{Nirjan Biswas$^1$, Ujjal Das$^{2,*}$, and Mrityunjoy Ghosh$^3$}


\email{nirjan22@tifrbng.res.in}
\email{ujjal.rupam.das@gmail.com}
\email{ghoshmrityunjoy22@gmail.com}

\thanks{$^*$Corresponding author.}

\begin{document}
\maketitle

 \centerline{$^{1}$Tata Institute of Fundamental Research, Centre For Applicable Mathematics,}
 \centerline{Post Bag No 6503, Sharada Nagar,}
 \centerline{Bangalore 560065, India}
 \centerline{$^2$Department of Mathematics, Technion - Israel Institute of
		Technology,   }
   \centerline{Haifa 32000, Israel}
   
   \centerline{$^{3}$Department of Mathematics, Indian Institute of Technology Madras,}
  \centerline{Chennai 600036, India}

\begin{abstract}
For $N\geq 2$, a bounded smooth domain $\Omega$ in $\mathbb{R}^N$, and  $g_0, V_0 \in L^1_{loc}(\Omega)$, we study the optimization of the first eigenvalue for the following weighted eigenvalue problem:
\begin{align*} 
 -\Delta_p \phi + V |\phi|^{p-2}\phi = \lambda g |\phi|^{p-2}\phi  \text{ in } \Omega, \quad \phi=0 \text{ on } \partial \Omega, 
\end{align*}
where $g$ and $V$ vary over the rearrangement classes of $g_0$ and $V_0$, respectively. We prove the existence of a minimizing pair $(\underline{g},\underline{V})$ and a maximizing pair $(\overline{g},\overline{V})$ for $g_0$ and $V_0$ lying in certain Lebesgue spaces.  
  We obtain various qualitative properties such as polarization invariance, Steiner symmetry of the minimizers as well as the associated eigenfunctions for the case $p=2$.
For annular domains, we prove that the minimizers and the corresponding eigenfunctions possess the foliated Schwarz symmetry.  
\end{abstract}

\section{Introduction}
Let $N \ge 2$ and $\Om$ be a smooth bounded domain in $\RN$. For $p \in (1, \infty)$ and $g,V\in L^1_{loc}(\Om)$, we consider the following weighted eigenvalue problem:
\begin{equation}\label{EVP}
\begin{aligned} 
 -\De_p \phi + V |\phi|^{p-2}\phi &= \la g |\phi|^{p-2}\phi  \text{ in } \Om,\\
 \phi&=0 \quad\qquad \;\;\;\text{ on } \pa \Om,
 \end{aligned}
 \end{equation}
where $\Dep \phi := \text{div}(|\nabla \phi|^{p-2}\nabla \phi)$ is the $p$-Laplace operator and $\la $ is a real parameter. We say $\la$ is an eigenvalue of \eqref{EVP}, if there exists $\phi \in \wp \setminus \{0\}$ so that the following identity holds:
\begin{align*}
    \int_{\Om} |\nabla \phi|^{p-2}\nabla \phi . \nabla \psi \, \dx & + \int_{\Om} V  |\phi|^{p-2}\phi \psi \, \dx 
     = \la \int_{\Om} g |\phi|^{p-2}\phi \psi \, \dx, \quad \forall \,  \psi \in \wp.
\end{align*}
Let $g,V$ be such that
$$ \La(g,V) := \inf \left\{ \frac{\int_{\Om} |\Gr \phi|^p + V|\phi|^p } { \int_{\Om} g |\phi|^p}  : \phi \in \wp, \int_{\Om} g |\phi|^p > 0 \right\} $$
is positive. If $\La(g,V)$ is attained at some $\phi \in \wp$, then we say
$\La(g,V)$ is the first eigenvalue and $\phi$ is a first eigenfunction of \eqref{EVP}. In the context of studying eigenvalue problems, many authors have provided various sufficient conditions on $g,V$ so that the first eigenvalue is simple (i.e., any two first eigenfunctions are constant multiple of each other), and principal (i.e., first eigenfunctions do not change their sign). For example, we refer \cite{AMM, Cuesta2009, Leadi, Szulkin} to list a few and the references therein. In this article, we make the following assumptions on $g$ and $V$:
\begin{equation}\tag{\textbf{A}}\label{A}
\begin{aligned}
& (\textbf{A1}) \quad g,V \in X:= \left\{ \begin{array}{ll}
 L^{\frac{N}{p}}(\Om), &  \text{ if }  N > p; \vspace{0.1 cm} \\
             L^q(\Om); q \in (1, \infty), &  \text{ if }  N\leq p, 
              \end{array}\right. \\
& (\textbf{A2}) \quad g^+ \not\equiv 0 \text{ and } \norm{V^{-}}_{X} \le \displaystyle \frac{1-\de_0}{S^p}, \text{ for some } \delta_0 \in (0, 1),
\end{aligned}
\end{equation}
where for a function $f:\Om\mapsto \mathbb{R}$, $f^{\pm}(x) := \max\{ \pm f(x), 0 \}$, $S$ is the embedding constant of $\wp \hookrightarrow L^{p^*}(\Om)$ ($p^*= \frac{pN}{N-p}$) if $N > p$ and $\wp \hookrightarrow L^{pq'}(\Om)$  if $N \leq p$. For $g, V$ satisfying \eqref{A}, using variational technique and the Picone's identity, one can  show that $\La(g,V)$ (for instance, see \cite{Leadi} when $\Om$ is unbounded) is a simple principal eigenvalue of \eqref{EVP}. However, for the sake of completeness, we give a proof of these results for bounded domains in the Appendix (Theorem \ref{Existence1}). Now,
for $g_0, V_0$ as given in \eqref{A}, we define:
\begin{align*}
 \La_{\min}(g_0,V_0) &=   \inf \left\{ \La(g,V): g \in \E(g_0), V \in \E(V_0) \right\}, \\
 \La_{\max}(g_0,V_0) &=   \sup \left\{ \La(g,V): g \in \E(g_0), V \in \E(V_0) \right\},
\end{align*}
where $\E(f)$ is the set of all rearrangements of a measurable function $f$, which is defined as 
$$\E(f)=\{h: \Om \mapsto \R: \, h \ \mbox{is measurable}, \, |\{h>t\}| = |\{f>t\}|, \; \forall \, t \in \R\}.$$ 
In this article, we are concerned about the following optimization problems:
\begin{align}
    \text{\textit{does there exist}}\;(\ug,\uv)\in \E(g_0)\times \E(V_0)\;\text{\textit{such that}}\;\La_{\min}(g_0,V_0)=\La(\ug,\uv)?\label{OurProb_min}\\
    \text{\textit{does there exist}}\;(\og,\ov)\in \E(g_0)\times \E(V_0)\;\text{\textit{such that}}\;\La_{\max}(g_0,V_0)=\La(\og,\ov) \label{OurProb_max}?
\end{align}

The above problems have significant importance in the study of elasticity theory, population dynamics, and many other mathematical models. For example, the following diffusive logistic equation is considered in \cite{Skellam}:
\begin{equation}\label{logistic}
    \begin{aligned}
    u_t - \De u &= \mu (g_0-u)u \quad \mbox{in} \; \Om\times (0,\infty),  \\
    u &= 0 \quad   \mbox{on} \; \partial \Om\times (0,\infty), \\
    u(x,0) & \gneqq 0 \quad  \mbox{on} \;  \overline{\Om}, 
    \end{aligned}
\end{equation}
where $u(x,t)$ represents the density of a species at position $x$ and time $t$, $g_0$ is a weight function, $\mu$ is  a positive parameter, and $u=0$ on $\partial \Om\times (0,\infty)$ (i.e., Dirichlet conditions) represents that the region outside the domain is completely lethal. In this mathematical model, one can predict the persistence or extinction of a species by means of certain parameters that are directly related to the principal eigenvalue of Laplacian \cite{Cantrell,Robert}. More precisely, in \eqref{logistic}, $\mu > \La_{\max}(g_0,0)$ ensures the survival of the species and $\mu \leq \La_{\min}(g_0, 0)$ leads to the extinction of the species as time $t$ increases. In this viewpoint, it is important to identify an optimal distribution of resources that optimize $\Lambda(g,0)$ over the rearrangement class.  
Also, studying qualitative properties of such optimizers (if it exists) is equally essential to know the nature of an optimal arrangement, such as the shape of a favorable and unfavorable zone for the species to survive, fragmentation/concentration phenomena, etc. To see more such problems, see \cite{Berestycki2005, Mazari2020} and the references therein.

The objective of this article is twofold. Firstly, we study the existence of optimizers in \eqref{OurProb_min}-\eqref{OurProb_max} for a general class of weight functions and potentials as mentioned in \eqref{A}. Secondly,  we investigate the geometry of the optimizers.

\subsection{Existence of optimizers}
Krein \cite[for $N=1,p=2$]{Krein} and  Cox-McLaughlin \cite[$N\geq 1$, $p=2$]{Cox} have studied the optimization problems \eqref{OurProb_min}-\eqref{OurProb_max} for $V_0=0,$ $g_0= \al \chi_D + \beta \chi_{\Om\setminus D}$, $0\leq \al<\beta$ and $D \subset \Om$ with $0 < |D|=c < |\Om|$, where the optimization was taken over the class
$$\mathcal{A}_{\al,\beta,c}=\left\{g \in L^{\infty}(\Om): \al \leq g \leq \beta, \int_{\Om} g = c\right\}.$$ Several authors have considered similar problems where the optimization parameter varies over different admissible classes, e.g. \cite{Ashbaugh, Bonder,  Pielichowski, Takac2010}.  The authors of \cite{Cuccu} considered the optimization problems \eqref{OurProb_min}-\eqref{OurProb_max} over the rearrangement class $\mathcal{E}(g_0)$. If $V_0=0$ and $g_0 \in L^{\infty}(\Om)$ with $g_0\geq 0$, they have proved that $\La_{\min}(g_0,0)$ and $ \La_{\max}(g_0,0)$ are attained in $\mathcal{E}(g_0)$ and $ \overline{\mathcal{E}(g_0)}$ (weak* closure of $\mathcal{E}(g_0)$ in $L^{\infty}(\Om)$) respectively. In addition, if $\Om$ is a ball, they showed that $ \La_{\max}(g_0,0)$ is attained in $\mathcal{E}(g_0)$ itself. For the minimization problem \eqref{OurProb_min}, in \cite{Prajapat}, authors consider nonnegative $g_0, V_0 \in L^{\infty}(\Om)$ with certain restriction on the norm of $V_0$. In \cite{Leandro}, authors have enlarged the class of weight functions and potentials for the existence of minimizer by considering $g_0,V_0 \in L^q(\Om)$ with $q > \frac{N}{p}$ (if $N\geq p$), and $q=1$ (if $N<p$), and certain restriction on  $\norm{V_0^-}_{L^q}$. In order to get the existence of optimizers, 
the compactness of the Sobolev embedding $\wp \hookrightarrow L^r(\Om)$ with $r<p^*$ (if $N >p$) and $r< \infty$ (if $N \le p$) plays an important role. In this article, we extend all the existence results for \eqref{OurProb_min}-\eqref{OurProb_max} with $g_0,V_0$ satisfying \eqref{A}. Now we state our results.
\begin{theorem}\label{optimization}
Let $\Om$ be a smooth bounded domain in $\R^N$.  Assume that $g_0, V_0$ satisfy \eqref{A}. Then the following holds:
\begin{enumerate}
    \item[(i) \rm{{(\textbf{Existence of minimizer})}}] There exists $(\ug ,\uv)\in \E(g_0) \times \E(V_0)$ such that
$$\La_{\min}(g_0, V_0) = \La(\ug,\uv),$$
\item[(ii) \rm{{(\textbf{Existence of maximizer})}}] In addition, if $g_0\geq 0$, then there exists $(\og, \ov) \in \E(g_0) \times \E(V_0)$ such that
$$ \La_{\max}(g_0, V_0) = \La(\og,\ov).$$
\end{enumerate}
\end{theorem}

For $N>p$ and $g_0,V_0 \in L^{\frac{N}{p}}(\Om)$, one of the main difficulties occurs in the minimization problem due to the non-compactness of the critical Sobolev embedding $\wp \hookrightarrow L^{p^*}(\Om)$. However, we overcome this by using certain regularity of the solution of \eqref{EVP} due to Guedda-Veron \cite{Guedda} and a gradient estimate obtained by Damascelli-Pardo \cite{Lucio}. For the maximization problem, we mainly use the rearrangement inequality (by Burton \cite{Burton}) to get the existence of maximizer in the rearrangement classes of $g_0,V_0$.

In Theorem \ref{optimization}, we call each of $(\uphi,\ug,\uv)$ and $(\ophi,\og,\ov)$ as an \textit{optimal triple}, where $\uphi$ is a first eigenfunction of \eqref{EVP} associated to $\ug,\uv$ and $\ophi$ is a first eigenfunction of \eqref{EVP} associated to $\og,\ov$. Notice that if $g_0, V_0$ are constant functions, then the rearrangement class is singleton.   
In these cases, we call an optimal triple as \textit{optimal pair}. We set 
\begin{align*}
    (\uphi,\ug,\uv):= \left\{ \begin{array}{ll}
        (\uphi,\ug), &  \text{ if $V_0$ is constant}, \vspace{0.1 cm} \\
        (\uphi,\uv), &  \text{ if $g_0$ is constant}, 
              \end{array}\right. 
\end{align*}
and the similar convention holds for $(\ophi,\og,\ov)$ as well.

\subsection{Symmetry of minimizers}
In the pioneering article \cite{Chanillo}, authors considered \eqref{OurProb_min} for $p=2$, $g_0=1$ and $V_0=\al\chi_D$, where $\al>0$ and $\chi_D$ is the characteristic function of a measurable set $D\subset \Om$, and proposed the question of whether, or not, an optimal pair $(\uphi,\uv)$ inherits the symmetry of the underlying domain $\Om$.  In \cite[Theorem 4]{Chanillo}, they proved that if $\Om$ is Steiner symmetric with respect to a hyperplane $P$ (i.e., $\Om$ is convex and symmetric with respect to $P$), then $(\uphi,\uv)$ is also Steiner symmetric with respect to $P$. They also  showed that a symmetry of the underlying domain  would not carry to an optimal pair $(\uphi,\uv)$ without the convexity assumption. For example,   \cite[Theorem 6]{Chanillo} provides a concentric annular region and a $V_0$ for which $(\uphi,\uv)$ is not rotationally symmetric. Furthermore, \cite[Theorem 7]{Chanillo} gives a dumbbell domain  for which the axial symmetry breaks for an optimal pair. In \cite[Section 6]{Chanillo}, authors have also conjectured several necessary and sufficient criteria on domains (concentric annulus, dumbbell, etc.) for which symmetry is preserved. 

For  certain  convex domains, the minimizers of \eqref{OurProb_min}  preserve the symmetry of the underlying  domains. For example, when $\Om=B_1(0)$, with the same assumptions on $g_0$ and $V_0$ as in \cite{Chanillo}, Pielichowski \cite{Pielichowski} proved that an optimal pair $(\uphi,\uv)$ is radial. For $V_0 \equiv 0$ and nonnegative $g_0\in L^\infty(B_1(0))$, in \cite{Cuccu} authors showed that an optimal pair $(\uphi,\ug)$ is radial and radially decreasing in $B_1(0)$. 
This result has been further extended by Emamizadeh-Prajapat \cite[Theorem 3.3]{Prajapat} for nonnegative $V_0 \in L^{\infty}(B_1(0))$ satisfying certain norm bound, and in addition, authors obtained that $\uv$ is radial and radially increasing in $B_1(0)$. For Steiner symmetric domains, the Steiner symmetry of  $(\uphi,\ug)$  is obtained in \cite[Theorem 3.1]{Anedda} for $g_0=\al \chi_D + \beta \chi_{\Om\setminus D}$ (where $0\leq \al<\beta$ and $D\subset \Om$ such that $0 < |D| < |\Om|$) and $V_0=0.$ For similar symmetry preserving results related to other variational problems in this direction, we refer to \cite{Cadeddu2011Porru,Cuccu2002Porru,Porru2011Jha, Kurata2004} and the references therein. We also refer to \cite{Lamboley,Mazzolini} for further   results  on the  symmetry of the optimal weights.

Notice that, for the domains where symmetry breaking  happens, the classical symmetrizations such as Schwarz and Steiner symmetrization were not applicable. However, it is natural to ask: for such domains, do optimal pairs have any  partial symmetry? In this article, using \textit{polarization} (also known as \textit{two-point symmetrization}; cf. \cite{Bianchi2020,Brock2000}), we prove various symmetries of an optimal triple on a more general class of domains (not necessarily simply connected) for the linear case (i.e., $p=2$ in \eqref{OurProb_min}). To the best of our knowledge, there are no such results available in the literature regarding the symmetry properties of an optimal triple for the problem \eqref{OurProb_min} on domains that are not simply connected, except a few counterexamples (for symmetry breaking) mentioned
earlier. Before stating our results, we first define polarization of a domain and polarization of a function.

\subsection*{Polarization} Let $H$ be an open affine half-space in $\R^N$ and $\sigma_H$ denote the reflection with respect to the boundary $\partial H$ of $H$.
\begin{definition}\label{pola_def}
 (i) The \textit{polarization} of $\Om\subset \R^N$ with respect to $H$ is defined as
\begin{equation*}
    \Om_H =\big((\Om\cup\sigma_H(\Om))\cap H\big)\cup\big(\Om\cap \sigma_H(\Om)\big).
\end{equation*}

    (ii)\label{Polarization_function}
 For a measurable function $f:\mathbb{R}^N \rightarrow \R$, the \textit{polarization} of $f$ with respect to $H$ is defined as
\begin{align*}
    f_H(x)=\left\{\begin{array}{cc}
        \max \{f(x),f(\sigma_H(x))\}, & \text{if} \ x \in H, \\
        \min \{f(x),f(\sigma_H(x))\}, & \ \text{if}\  x \notin H.
    \end{array} \right.
\end{align*}
For $\Om\subsetneq\R^N$, we define the \textit{polarization} of a function $f:\Om \rightarrow \R$ with respect to $H$ by $f_H=\tilde{f}_H|_\Om$, where $\tilde{f}$ is the extension of $f$ to $\mathbb{R}^N$ by 0 outside of $\Om$. We also define a \textit{dual-polarization} of $f$ as $f^H =f_H\circ\sigma_H.$

(iii) Let $f:\Om \rightarrow \R$ be a measurable function. If $ f_H=f$ a.e. in $\Om$, then $f$ is said to be \textit{polarization invariant} with respect to $H$. Similarly, if $f^H=f$ a.e. in $\Om$, then $f$ is said to be \textit{ dual-polarization invariant} with respect to $H$.
\end{definition}

Now we state our next result.
\begin{theorem}\label{min_thm}
Let $p=2$ and $H\subset \R^N$ be an open  affine half-space such that $0\in \overline{H}$. Let $\Om$ be a smooth, bounded domain in $\R^N$ such that $\Om=\Om_H$. Let $g_0, V_0$ satisfy the assumption as given in \eqref{A}. In addition, we assume that $g_0, V_0 \ge 0$. Let $(\uphi,\ug,\uv)$ be an optimal triple as given by Theorem \ref{optimization}-$(i)$. Then the following holds:
\begin{enumerate}
    \item [$(i)$]
   if $\sigma_H(\Om)\neq\Om$ and $V_0=0$, then $\uphi,\; \ug$ are polarization invariant with respect to $H$,
   \item [$(ii)$]
    if $\sigma_H(\Om)=\Om$, then $\uphi,\ug,\uv$ are either polarization invariant or else dual-polarization invariant with respect to $H$.
    \end{enumerate}
\end{theorem}

Let us now briefly describe the technique of our proof. As seen in  \cite{Anedda, Chanillo, Cuccu}, the techniques for proving the Schwarz and Steiner symmetry of the minimizers mainly  rely on the Hardy-Littlewood inequality and the characterizations for the equality case in P\'{o}lya-Szeg\"{o} inequality, namely, $(i)$ Brothers-Ziemer's characterization \cite[for Schwarz symmetrization]{Brothers}, $(ii)$ the counterpart of Brothers-Ziemer's characterization due to Cianchi-Fusco \cite[for Steiner symmetrization]{Cianchi}. Indeed, an analogue of the Hardy-Littlewood inequality for polarization plays a vital role in our proof as well. However, since the gradient norm of a function remains unchanged under polarization (Proposition \ref{pola_bounded}), equality occurs in the P\'{o}lya-Szeg\"{o} type inequality. Thus the analogue of Brothers-Ziemer type characterization is no more valid in the case of polarization. We bypass this deficiency by using a version of strong maximum principle (Proposition \ref{STM}) and  compare $\uphi$ and $\uphi_H$ on $\Om\cap H$. This indeed helps us to prove the above theorem.

As we mentioned earlier, for $g_0$ taking a finite number of nonnegative values, Anedda-Cuccu studied the Steiner symmetry of minimizers \cite[Remark 3.1]{Anedda}. This particular choice of $g_0$ allowed them to use the result by Cianchi-Fusco \cite [Theorem 2.6]{Cianchi} in their proof. In this article, as an application of Theorem \ref{min_thm}, we extend Cuccu-Anedda's result for a more general class of weight functions $g_0$.

\begin{corollary}[Steiner symmetry]\label{Steiner_theo}
Let $p,g_0,H$ be as given in Theorem \ref{min_thm} and $V_0=0$. Assume that $\Om$ is a Steiner symmetric domain with respect to the hyperplane $\pa H$. Then an optimal pair $(\uphi,\ug)$ is Steiner symmetric with respect to $\pa H$ in $\Om$. In particular, we have $\uphi=\uphi\circ \sigma_H$ and $\ug=\ug\circ \sigma_H$ a.e. in $\Om$.
\end{corollary}

 We observe that the concentric annulus is polarization invariant with respect to any open half-space containing the origin on the boundary. On the other hand, the non-concentric annulus is polarization invariant with respect to any open half-space which contains the origin on the boundary and does not contain the center of the inner ball. This kind of geometry motivates us to study certain partial symmetry of  $(\uphi,\ug,\uv)$ on the annular region. Indeed, in the following theorem, we show that $(\uphi,\ug, \uv)$ is foliated Schwarz symmetric in annular domains. 

\begin{theorem} \label{Foliation_minimization}
Let $\Om_{R,r}=B_R(0)\setminus \overline{B_r(te_1)},\;0\leq t<R-r$ and $p,g_0,V_0$ be as in Theorem \ref{min_thm}. Let $(\uphi,\ug,\uv)$ be an optimal triple. Then the following holds:
\begin{enumerate}
    \item [$(i)$ \rm{(\textbf{Concentric case})}] if $t=0$, then there exists $\ga\in \S^{N-1}$ such that $\uphi,\ug$ are foliated Schwarz symmetric in $\Om_{R,r}$ with respect to $\ga$ and $\uv$ is foliated Schwarz symmetric in $\Om_{R,r}$ with respect to $-\ga$,
    \item [$(ii)$ \rm{(\textbf{Non-concentric case})}] if $t>0$ and $V_0=0$, then $\uphi$ and $\ug$ are foliated Schwarz symmetric in $\Om_{R,r}$ with respect to $-e_1$.
\end{enumerate}
\end{theorem}

As a by-product of Theorem \ref{min_thm} and Theorem \ref{Foliation_minimization}, we prove that maxima of the first eigenfunction of \eqref{EVP} associated to a minimizer of \eqref{OurProb_min} on nonconcentric annulus  will lie on a segment of the negative $x_1$-axis.
\begin{corollary} \label{Maxima_location}
Let $p=2$ and $\Om=\Om_{R,r}=B_R(0)\setminus \overline{B_r(te_1)}$, where $0< t<R-r$. Assume that $g_0\in L^q(\Om)$, where $q>\frac{N}{2}$, is nonnegative and $V_0=0$. Let $(\uphi,\ug)$ be an optimal pair. Define $$L_{\Om}=\bigg\{x=\big(x_1,x_2,\dots,x_N\big)\in \Om\cap (-\R^+ e_1): x_1\geq -\frac{R+r-t}{2}\bigg\},$$
where $\R^+$ is the set of nonnegative real numbers. Then $\max\limits_{x\in \Om}\uphi(x)=\max\limits_{x\in L_{\Om}} \uphi(x)$. In addition, if $\ug$ is continuous, then $\max\limits_{x\in \Om}\ug(x)=\max\limits_{x\in L_{\Om}} \ug(x)$.
\end{corollary}

The remainder of the article is organized as follows.  In Section \ref{pola_section}, we briefly discuss polarization and prove certain related results that are essential for the development of this article. In Section \ref{symmetriz_section}, we recall three different types of symmetrizations and their characterizations in terms of polarization. Proof of the existence result (Theorem \ref{optimization}) is given in Section \ref{existence_section}. In Section \ref{symmetry_section}, we study the symmetry results. This section contains the proof of Theorem \ref{min_thm}-\ref{Foliation_minimization} and Corollary \ref{Steiner_theo}-\ref{Maxima_location}. The existence of the first eigenvalue of \eqref{EVP} is derived in Appendix.

\section{Preliminaries}
\subsection{Polarizations}\label{pola_section}
Let $\H$ be the collection of all open affine half-spaces in $\R^N$, and $\H_0\subset\H$ denotes the set of all $H\in\H$ such that $0\in \overline{H}$. For $\beta \in \R^N$, we set
\begin{align*}
    \widehat{\H}_0:=\{H\in \H_0: 0\in \pa H\}, \; \, \H(\beta):=\{H\in \H:\beta\in H\}, \; \, \widehat{\H}_0(\beta):=\{H\in\widehat{\H}_0 :\beta\in H\}.
\end{align*}

In the next proposition, we prove some results which will be used in subsequent sections. 
\begin{proposition}\label{fact_pola_domain}
Let $H\in \H$ and $\Om$ be a domain in $\R^N$ such that $\Om=\Om_H$. Then 
\begin{enumerate}[(i)]
    \item $\sigma_H( \Om^c\cap H) \subset \Om^c\cap H^c$.
    \item $\sigma_H(\Om \cap \overline{H}^c) \subset \Om \cap H$.
    \item if $\sigma_H(\Om)\neq \Om$, then there exists $A\subset \Om\cap H$ such that $|A|>0$ and $\sigma_H(A)\subset \Om^c\cap  \overline{H}^c$.
    \item Let $f:\Om \rightarrow \R^+$ be a measurable function. Let $\tilde{f}_H$ be the polarization of $\tilde{f}$ as given in Definition \ref{pola_def}-(ii). Then $\tilde{f}_H=0$ a.e. in $\Om^c$.
\end{enumerate}

\end{proposition}
\begin{proof}

$(i)$ Let $x\in \Om^c\cap H$. Then $\sigma_H(x)\in H^c$. We claim that $\sigma_H(x)\in \Om^c$. On the contrary, suppose $\sigma_H(x)\in \Om$. Let $y=\sigma_H(x)$. Then $\sigma_H(y)\in \sigma_H(\Om)$. Thus $\sigma_H(y)\in \sigma_H(\Om)\cap H\subset \Om_H$. Since $\Om=\Om_H$, we have $\sigma_H(y)\in \Om$. Therefore $x(=\sigma_H(y))\in \Om$, which is a contradiction as $x\in \Om^c$.

$(ii)$ Proof follows using a similar set of arguments as given above.

$(iii)$ From $(ii)$,  we have $\sigma_H(\Om \cap \overline{H}^c) \subset \Om \cap H$. Since $\sigma_H(\Om)\neq \Om$, we get $\sigma_H(\Om \cap \overline{H}^c) \subsetneq \Om \cap H$. Therefore, the set $A:=(\Om \cap H)\setminus \sigma_H(\Om \cap \overline{H}^c)$ is nonempty. Then $\sigma_H(A)\subset \Om^c\cap \overline{H}^c$. 
Now it is enough to show that int$(A)$ is nonempty. Suppose int$(A)=\emptyset$. Then for every $x\in A$, there exists $r_x>0$ such that  $B_r(x)\cap \sigma_H(\Om \cap \overline{H}^c)\neq \emptyset, \; \forall \, r\in (0,r_x)$. This implies that $A\subset \pa(\sigma_H(\Om \cap \overline{H}^c))$ and hence
\begin{align}\label{int}
    \sigma_H(\Om \cap \overline{H}^c) \subsetneq \Om \cap H\subset \overline{\sigma_H(\Om \cap \overline{H}^c)}.
\end{align}
On the other hand, $A \subset \pa(\sigma_H(\Om \cap \overline{H}^c)) \cap (\Om\cap H)$ and $\Om\cap H$ is open. Hence for $y\in A$, there exists $r>0$ such that $B_r(y)\subset \Om\cap H$ and $B_r(y)\cap \left(\overline{\sigma_H(\Om \cap \overline{H}^c)}\right)^c\neq \emptyset$, a contradiction to \eqref{int}. Thus, int$(A)$ must be nonempty.

$(iv)$ Let $x\in  \Om^c\cap H$. Since $\Om=\Om_H$, using Proposition \ref{fact_pola_domain}-$(i)$, $\sigma_H(x)\in \Om^c\cap H^c$ and  $\tilde{f}_H(x)= \max \{\tilde{f}(x), \tilde{f}(\sigma_H(x))\} = 0.$ If $x \in \Om^c\cap H^c$, then $\tilde{f}_H(x)=\min \{ \f(x), \f(\sigma_H(x)) \} \leq 0$. Thus, $\tilde{f}_H=0$ a.e. in $\Om^c$.
\end{proof}

In the next proposition, we prove that the polarization of a measurable function defined on $\Om$ is a rearrangement of that function. For $\Om=\R^N$, this result is well known as polarization is a two-point rearrangement (see \cite[Section 5]{Brock2000}). For $\Om\subsetneq\R^N$, we give a proof using Definition \ref{pola_def}-(ii). We also state some results related to the invariance of norms under polarization.
\begin{proposition}\label{pola_bounded}
Let $H\in \H$ and let $\Om\subsetneq\R^N$ be a domain such that $\Om=\Om_H$. Let $f:\Om \rightarrow \R^+$ be a measurable function, and its polarization $f_H$ be as given in Definition \ref{pola_def}-(ii). Then the following holds:
\begin{enumerate}[(i)]
    \item $f_H$ is a rearrangement of $f$,

    \item  If $f\in L^p(\Om)$ for some $p\in [1,\infty)$, then $f_H\in L^p(\Om)$ with $\norm{f}_p=\norm{f_H}_p$. Furthermore, if $f\in \wp$, then $f_H\in \wp$ with $\norm{\nabla f}_p=\norm{\nabla f_H}_p$.
\end{enumerate}
\end{proposition}
\begin{proof} 

$(i)$ Let $t<0$. Since $f \ge 0$, it is clear that $\tilde{f}_H \ge 0$. Thus $f_H \ge 0$ and hence $\abs{\{x \in \Om: f_H(x)>t\}}=|\Om|$. Let $t \ge 0$. In this case, it is easy to observe that
\begin{align}\label{eq1}
\abs{\{x\in \Om:f(x)>t\}} & = \abs{\{x\in \R^N:\tilde{f}(x)>t\}} \no \\
&= \abs{\{x\in \R^N:\tilde{f}_H(x)>t\}} \no \\
& =\abs{\{x\in \Om:f_H(x)>t\}}+\abs{\{x\in \Om^c:\tilde{f}_H(x)>t\}}.
\end{align}
Since $f \ge 0$ a.e. in $\Om$, applying Proposition \ref{fact_pola_domain}-$(iv)$ we have $\abs{\{x\in \Om^c:\tilde{f}_H(x)>t\}} = 0$. Therefore, from \eqref{eq1} we conclude $\abs{\{x\in \Om:f(x)>t\}}=\abs{\{x\in \Om:f_H(x)>t\}}.$

$(ii)$ Both the claims follow from \cite[Proposition 2.3]{Scaftingen2005}.
\end{proof}

In the following remark, we enlist some elementary  facts about the polarized domains and functions. If $g=h$ a.e. in $\Om$, then we write $g=h$ in $\Om$ now onwards.
\begin{remark}\label{Facts_polarization}
Let $H\in \H, \Om\subset\R^N$ be a domain and $f:\Om\rightarrow\R^+$ be a measurable function.
\begin{enumerate}[(i)]
    \item \label{Reflection_symmetry} If $\Om=\Om_H=\Om_{\overline{H}^c}$, then $\Om$ is symmetric with respect to the hyperplane $\pa H$. For such domain if $f$ satisfies $f=f_H=f^H$ in $\Om$, then it is easy to see that $f=f\circ \sigma_H$ in $\Om$, i.e., $f$ is symmetric with respect to $\pa H$.
    
    \item \label{Upper} From Definition \ref{pola_def}-(ii), it follows that 
    \begin{align*}
        f^H =f_{\overline{H}^c},\quad f_H=f^{\overline{H}^c}, \quad (f^H)_H=(f_H)_H=f_H,\quad (f_H)^H=(f^H)^H=f^H.
    \end{align*}
    \item \label{Equimeasurability} 
    If $\Om=\Om_{\overline{H}^c}$, then (analogous to Proposition \ref{pola_bounded}-$(i)$), $f^H$ is a rearrangement of $f$. However, the assumption $\Om=\Om_H$ alone is not sufficient to ensure that $f^H$ is a rearrangement of $f$. For example, we consider an open set  $\Om := \{x\in \R^2: |x|<1\}\cap \{x \in \R^2: x_2>0 \}$. Let $f: \Om \rightarrow \R^+ \setminus \{0\}$ be a measurable function. Let $H\in \H_0(e_2)$ where $e_2 = (0,1)$. Then $\Om= \Om_H$ and $\sigma_H(\Om) \neq \Om$. Therefore,  $\sigma_H(\Om \cap \overline{H}^c) \subsetneq \Om \cap H$ (by Proposition \ref{fact_pola_domain}-$(ii)$). Set $A= (\Om \cap H) \setminus \sigma_H(\Om \cap \overline{H}^c)$. From Proposition \ref{fact_pola_domain}-$(iii)$, $|A|>0$. Then for each $x \in A, \sigma_H(x) \in \Om^c \cap H^c$ and hence using Definition \ref{pola_def}-(ii),  $f_H(x) = 0$. Thus $\abs{x \in \Om: f_H(x)>0} \le |\Om \setminus A| < |\Om| = \abs{x \in \Om: f(x)>0}$.
    
    \item If $\Om=\Om_H$ and $f\in \hh2$, then it is not necessary that $f^H$ lies in $\hh2$. For example, consider $\Om \subset \R^2$ and $H$ as above. For such $H$, it is easy to see that $f^H\notin \hh2$. However, in addition if $\Om=\Om_{\overline{H}^c}$, then we have $f^H=f_{\overline{H}^c}\in \hh2$.
    
\end{enumerate}
\end{remark}

\subsubsection{Hardy-Littlewood and reverse Hardy-Littlewood inequality}
Next, we discuss the Hardy-Littlewood and the reverse Hardy-Littlewood inequality for polarization. 
\begin{proposition}\label{Hardy_RN}
Let $p \in (1,\infty)$, $H\in \H_0$ and $v,w\in L^p(\R^N)$ be such that $vw\in L^1(\R^N)$. Then
$$\int_{\R^N} v(x) w(x) \, \dx \leq \int_{\R^N} v_H(x) w_H(x)\, \dx.$$
\end{proposition}
\begin{proof}
 For a proof, we refer to \cite[Lemma 2]{BrockF}.
\end{proof}

In the following proposition, we first derive the Hardy-Littlewood inequality for functions defined on polarization invariant domains other than $\R^N$. Then, we prove a reverse Hardy-Littlewood inequality involving the polarization and the dual-polarization of functions.

\begin{proposition}\label{Hardy_Littlewood} Let $p \in (1, \infty)$, $H\in \H_0$ and $\Om\subset \R^N$ be  a bounded domain such that $\Om=\Om_H$. Let $v,w \in L^p(\Om)$ with $vw \in L^1(\Om)$. 
\begin{enumerate}
 \item [$(i)$ \rm{(Hardy-Littlewood inequality)}] Assume that at least one of $v$ and $w$ are nonnegative. Then 
\begin{align}\label{H-L}
    \int_{\Om} v(x) w(x) \, \dx \leq \int_{\Om} v_H(x) w_H(x) \, \dx.
\end{align}
\item [$(ii)$ \rm{(Reverse Hardy-Littlewood inequality)}] Assume that $w$ is nonnegative. Then
$$\int_{\Om} v^H(x) w_H(x) \, \dx \leq \int_{\Om} v(x)w(x) \, \dx.$$
\end{enumerate}
\end{proposition}
\begin{proof} $(i)$ Let $\tilde{v},\tilde{w}$ are the zero extensions of $v,w$ respectively to $\mathbb{R}^N$. Then, using Proposition \ref{Hardy_RN}, we have 
\begin{equation}\label{in_2}
    \int_{\R^N}  \tilde{v}(x) \tilde{w}(x) \, \dx \leq  \int_{\R^N} \tilde{v}_H(x) \tilde{w}_H(x) \, \dx.
\end{equation}
From the definition $\tilde{v}(x)= 0$ for $x \in \Om^c$. Using \eqref{in_2}, we write
\begin{align*}
    \int_{\Om} v(x)w(x) \, \dx=\int_{\R^N}  \tilde{v}(x)\tilde{w}(x) \, \dx & \leq \int_{\R^N} \tilde{v}_H(x) \tilde{w}_H(x) \, \dx \no \\
    & = \int_{\Om} v_H(x)w_H(x) \, \dx +\int_{\Om^c} \tilde{v}_H(x) \tilde{w}_H(x) \, \dx.
\end{align*}
Without loss of generality, we assume $v \ge 0$ in $\Om$. Applying Proposition \ref{fact_pola_domain}-$(iv)$ we see that $\tilde{v}_H= 0$ in $\Om^c$. Thus from the above inequality, we get \eqref{H-L}.

$(ii)$ First, notice that
\begin{align*}
    -v^H(x)=-v_H(\sigma_H(x))
    =(-v)_H(x).
\end{align*}
Now, using \eqref{H-L}, we get
$$\int_{\Om}(-v)(x) w(x) \, \dx \leq \int_{\Om}(-v)_H(x) w_H(x) \, \dx= -\int_{\Om}v^H(x)w_H(x) \, \dx.$$ 
Therefore, $ \int_{\Om} v^H(x) w_H(x) \, \dx \leq \int_{\Om} v(x) w(x) \, \dx.$ 
\end{proof}

\subsection{Symmetrizations}\label{symmetriz_section}
In this section, we define Schwarz symmetry, Steiner symmetry, and Foliated Schwarz symmetry of a function. We also characterize these symmetries using polarization.  
\subsubsection{Schwarz symmetry}
\begin{definition}[Schwarz symmetric function]\label{Schwarz_func}
Let $f:B_1(0)\rightarrow\R$ be a measurable function. Then $f$ is called \textit{Schwarz symmetric} in $B_1(0)$ if $f$ is radial and radially decreasing in $B_1(0)$.
\end{definition}

Now we give an equivalent criterion for Schwarz symmetry via polarization. The following result is proved in \cite[Lemma 6.3]{Brock2000}.
\begin{proposition}\label{Schwarz_char}
Let $f:B_1(0)\rightarrow\R$ be a measurable function. Then $f$ is Schwarz symmetric in $B_1(0)$ if and only if $f=f_H$ for all $H\in \H(0)$.
\end{proposition}

\subsubsection{Steiner symmetry} 
In this section, we give a definition of Steiner symmetrization; cf. \cite[Section 2.2]{Henrot2006}. First, we fix some notations. We write $x \in \RN$ as  $x=(x',x_N)$, where $x'=(x_1,x_2,\dots,x_{N-1})\in \R^{N-1}$ and $x_N\in \R$. Let $\pi_{N-1}$ denotes the orthogonal projection from $\R^N$ to $\R^{N-1}$. For a measurable set $\Om\subset \R^N$, we define the slice of $\Om$ through $x'$ in the direction $x_N$ as $\Om_{x'}=\{x_N\in \R:(x',x_N)\in\Om\}$.
\begin{definition}[Steiner symmetric domain]\label{Steiner_dom}
The \textit{Steiner symmetrization} of $\Om$ with respect to the hyperplane $x_N=0$ is defined by $$\Om^\#=\left\{(x',x_N)\in \R^N: |x_N|< \frac{|\Om_{x'}|_1}{2}, x'\in \pi_{N-1}(\Om)\right\},$$
where $|\cdot|_1$ denotes the $1$-dimensional Lebesgue measure. If $\Om=\Om^\#$ (up to translation), then $\Om$ is said to be \textit{Steiner symmetric} with respect to the hyperplane $x_N=0$ .
\end{definition}

 Equivalently $\Om$ is Steiner symmetric with respect to the hyperplane $x_N=0$ if $(i)$ $\Om$ is symmetric with respect to the hyperplane $x_N=0$, and $(ii)$ $\Om$ is convex with respect to the $x_N$-axis, i.e., any line segment parallel to the $x_N$-axis joining two points in $\Om$ lies completely inside $\Om$.

\begin{definition}[Steiner symmetric function]\label{Steiner_func}
Let $\Om\subset \mathbb{R}^N$ be a measurable set and $f:\Om\rightarrow\R$ be a nonnegative measurable function. Then the \textit{Steiner symmetrization} $f^\#$ of $f$ on $\Om^\#$ with respect to the the hyperplane $x_N=0$ is defined as
$$f^\#(x)=\sup\left\{c\in \R:x\in \{y\in \Om:f(y)\geq c\}^\#\right\}, \;\text{where}\;x\in \Om^\#.$$
Let $\Om=\Om^\#$. If $f=f^\#$ in $\Om$, then $f$ is called \textit{Steiner symmetric} with respect to the hyperplane $x_N=0$.
\end{definition}
 Next, we give a characterization of Steiner symmetric domains and Steiner symmetric functions in terms of polarization; cf. \cite[Lemma 6.3]{Brock2000}.
 \begin{proposition}\label{Steiner_char}
 Let $\Om$ be a measurable set in $\R^N$ and $f:\Om\rightarrow \R$ be a nonnegative measurable function. Also, let $\H_*\subset \H$ be the collection of all half-spaces $H$ such that $H$ contains the hyperplane
 $x_N=0$ and $\pa H$ is parallel to the hyperplane $x_N=0$. Then the following holds:
 \begin{enumerate}[(i)]
     \item $\Om=\Om^\#$ if and  only if $\Om=\Om_H$ for all $H\in \H_*$,
     \item if $\Om=\Om^\#$, then $f$ is Steiner symmetric with respect to the hyperplane $x_N=0$ if and  only if $f=f_H$ for all $H\in \H_*$.
 \end{enumerate}
 \end{proposition}
\subsubsection{Foliated Schwarz symmetry}
First, we define the foliated Schwarz symmetrization of a function on radial domains following \cite{Brock2020}.
\begin{definition}[Foliated Schwarz symmetrization]\label{Foliated_sz} Let $\Om$ be a radial domain with respect to 0 and  $f:\Om\rightarrow \R$ be a nonnegative measurable function. Then the \textit{foliated Schwarz symmetrization} $f^\cd$ of $f$ with respect to a vector $\beta\in \S^{N-1}$ is the function satisfying the following properties:
\begin{enumerate}[(i)]
     \item $f^\cd(x)=h(r,\theta)$, $\forall x\in \Om$, for some function $h:[0,\infty)\times [0,\pi)\rightarrow\R$, which is decreasing in $\theta$, where $(r,\theta):=\big(|x|,{\rm{arccos}}(\frac{x\cdot \beta}{|x|})\big)$.
    \item for $a,b\in\mathbb{R}$ with $a<b$ and $r\geq 0$, $$|\{x:|x|=r,\; a<f(x)\leq b\}|_{N-1}=|\{x:|x|=r,\; a<f^\cd(x)\leq b\}|_{N-1},$$
    where $|\cdot|_{N-1}$ denotes the $(N-1)$-dimensional Lebesgue measure.
\end{enumerate}
\end{definition}
\begin{definition}[Foliated Schwarz symmetric function]\label{Foliated_symmetry}
Let $\Om$ be a radial domain with respect to 0. Then a nonnegative measurable function $f:\Om\rightarrow \R$ is said to be \textit{foliated Schwarz symmetric} with respect to a vector $\beta\in \S^{N-1}$ if $f=f^\cd$.
\end{definition}
Next, we give an analogous definition of foliated Schwarz symmetry on nonconcentric annular domains motivated by \cite{Anoop2020Ashok}.
\begin{definition}[Foliated Schwarz symmetry on non-concentric annulus]\label{Foliated_ann}
Let $\Om_{R,r}=B_R(0)\setminus \overline{B_r(te_1)}$, where $0<t< R-r$, and $f:\Om_{R,r}\rightarrow\R$ be a nonnegative measurable function. We call $f$ is \textit{foliated Schwarz symmetric} with respect to $-e_1$ if $\tilde{f}$ is foliated Schwarz symmetric with respect to $-e_1$ in $B_R(0)$, where $\tilde{f}$ is the extension of $f$ to $B_R(0)$ by 0 outside of $\Om_{R,r}$.
\end{definition}

From the definition, it follows that if $f$ is foliated Schwarz symmetric with respect to $\beta\in \S^{N-1}$, then $f$ is axially symmetric with respect to the axis $\R \beta$ and decreasing in the polar angle $\theta={\rm{arccos}}\big(\frac{x\cdot\beta}{|x|}\big)$. Alternatively, this symmetry is also known as \textit{spherical symmetry} \cite{Kawohl1985} or \textit{co-dimension one symmetry} \cite{Brock2003} in the literature. Now we state a characterization for foliated Schwarz symmetry in terms of polarization. The first part of the following proposition is proved in \cite[Theorem 3.5]{Brock2020} for measurable functions. For continuous functions, the second assertion is proved in \cite[Proposition 2.4]{Weth2010}. However, using a similar approach as given in \cite[Theorem 3.5]{Brock2020}, one can obtain the same result for measurable functions. We omit the proof here.
\begin{proposition} \label{Foliated_char}
Let $p\in[1,\infty)$, $\Om$ be a radial domain with respect to 0 and $f\in L^p(\Om)$ be nonnegative. 
\begin{enumerate}[(i)]
    \item If for every $H\in \widehat{\H}_0$, either $f_H=f$ or $f^H=f$, then there exists $\ga\in \S^{N-1}$ such that $f$ is foliated Schwarz symmetric with respect to $\ga$.
    \item Let $\beta\in \S^{N-1}$. Then $f$ is foliated Schwarz symmetric with respect to $\beta$ if and only if $f_H=f$ for all $H\in \widehat{\H}_0(\beta)$.
\end{enumerate}
\end{proposition}
\begin{remark}\label{Foliated_noncon}
From Definition \ref{Foliated_ann} and Proposition \ref{Foliated_char}-$(ii)$, a nonnegative measurable function $f:\Om_{R,r}\rightarrow\R$ is foliated Schwarz symmetric with respect to $-e_1$ if and only if $\tilde{f}_H=\tilde{f}\;\text{in}\;B_R(0),\;\forall\,H\in \widehat{\H}_0(-e_1)$. Observe that by Definition \ref{Foliated_ann}, $\tilde{f}_H=\tilde{f}$ in $B_r(te_1)$ for all $H\in\widehat{\H}_0(-e_1)$. Therefore $f$ is foliated Schwarz symmetric in $\Om_{R,r}$ with respect to $-e_1$ if and only if $\tilde{f}_H=\tilde{f}\;\text{in}\;B_R(0)\setminus\overline{B_r(te_1)}$, i.e., $f_H=f\;\text{in}\;\Om_{R,r}$ for all $H\in \widehat{\H}_0(-e_1)$.
\end{remark}

\section{Existence of Optimizer}\label{existence_section}
In this section, we study the existence and uniqueness of both minimizer and maximizer for \eqref{OurProb_min}-\eqref{OurProb_max}. First, we recall a few properties of rearrangement and an important rearrangement inequality due to Burton \cite{Burton}.

\begin{proposition}\label{rearrangement}
Let $p \in [1, \infty)$ and $f_0 \in L^p(\Om)$.
\begin{enumerate}[(i)]
    \item If $f_1 \in \E(f_0)$, then $f_1^{\pm} \in \E(f_0^{\pm})$.
    \item If $f_1 \in \E(f_0)$, then $\norm{f_1}_p = \norm{f_0}_p$.
    \item Let $h \in L^{\p}(\Om)$. Then there exists $f_1, f_2 \in \E(f_0)$ such that 
\begin{align*}
    \dis f_1(x) h(x) \, \dx \le \dis f(x) h(x) \, \dx \le \dis f_2(x) h(x) \, \dx, \quad \forall \, f \in  \overline{\E(f_0)},
\end{align*}
where $\overline{\E(f_0)}$ is the weak closure of $\E(f_0)$ in $L^p(\Om)$.
\end{enumerate}

\end{proposition}
\begin{proof}
 $(i)$ It is enough to show that  for $t \in \R^+$, $|\{x \in \Om: f_1^{-}(x) >t\}| = |\{x \in \Om : f_0^{-}(x) > t\}|$. Let $t \in \R^+$. Then we have $\{ x \in \Om: f_i^-(x) > t \} = \{x \in \Om:f_i(x) < -t \}, \; i=0,1$. Therefore, as $f_1$ is a rearrangement of $f_0$, we get  
\begin{align*}
    |\{ x \in \Om: f_1^-(x) > t \}| = |\Om| - |\{x \in \Om:f_1(x) \ge -t \}|
    & = |\Om| - |\{x \in \Om:f_0(x) \ge -t \}| \\
    & = |\{ x \in \Om: f_0^-(x) > t \}|.
\end{align*}
Thus $f_1^- \in \E(f_0^-)$. In a similar procedure, $f_1^+$ is a rearrangement of $f_0^+$.

 $(ii)$ and $(iii)$ follow from  \cite[Lemma 2.1 and Lemma 2.4]{Burton}.
\end{proof}

The following proposition gives regularity and a gradient estimate of the solutions of \eqref{EVP} that play a crucial role in the existence of minimizer. 

\begin{proposition}\label{regularity} 
Let $p \in (1, \infty),  N\geq p,$ and $\Om$ be a bounded domain. 
\begin{enumerate}[(a)]
    \item Let $g, V \in L^{q}(\Om)$ with $q>\frac{N}{p}$. If $\phi \in \wp$ is a solution of \eqref{EVP}, then $\phi\in C^1(\overline{\Om})$.
    \item Let $ N>p$, and  $g, V \in L^{\frac{N}{p}}(\Om)$. Let $\phi \in \wp$ be a solution of \eqref{EVP}. Then 
\begin{enumerate}[(i)]
    \item $\phi \in L^r(\Om)$ for any $r \in [1, \infty)$.
    \item  there exists $C=C(N,r)>0$ such that $
     \norm{\Gr \phi }_{\frac{Nr(p-1)}{N-r}} \le C \norm{(\la g - V)|\phi|^{p-2}\phi}^{\frac{1}{p-1}}_{r}, \text{  for  } r \in [(p^*)', N).$
\end{enumerate}
\end{enumerate}
\end{proposition}

\begin{proof}
(a) Proof follows using \cite[Proposition 1.3]{Guedda} and \cite[Theorem 1]{Lieberman1988}. 

\noi (b) Proof of $(i)$ follows using \cite[Proposition 1.2]{Guedda}, and proof of $(ii)$ follows as a consequence of \cite[Theorem 2.7]{Lucio}.
\end{proof}

Next, we prove a preparatory lemma for Theorem \ref{optimization}.  

\begin{lemma}\label{convergence}
 Let $q,r \in (1, \infty)$. Let $f_n \rightharpoonup f$ in $L^q(\Om)$ and $h_n \rightarrow h$ in $L^{rq'}(\Om)$. Then 
\begin{align*}
    \lim_{n \rightarrow \infty} \dis f_n |h_n|^r = \dis f|h|^r.
\end{align*}
\end{lemma}

\begin{proof}
 Let $\ep > 0$ be given. For each $n \in \N$, we have
\begin{align*}
    \abs{\left(f_n|h_n|^r - f|h|^r \right)} \le |f_n-f||h|^r + |f_n|\abs{ \left(|h_n|^r-|h|^r \right)}.
\end{align*}
Since $f_n \rightharpoonup f$ in $L^q(\Om)$ and $|h|^r \in L^{q'}(\Om)$, there exists $n_1 \in \N$ such that
\begin{align}\label{cnvg1}
   \dis |f_n-f||h|^r < \ep, \quad \forall \, n \ge n_1.
\end{align}
Next, since $h_n \rightarrow h$ in $L^{rq'}(\Om)$, we get $\norm{|h_n|^r}_{q'} \ra  \norm{|h|^r}_{q'}$ and up to subsequence $|h_n|^r \rightarrow |h|^r$ a.e. in $\Om$. Hence $|h_n|^r \ra |h|^r$ in $L^{q'}(\Om)$. Therefore, there exists $n_2 \in \N$ such that
\begin{align}\label{cnvg2}
    \dis |f_n|\abs{ \left(|h_n|^r-|h|^r \right)} \le \norm{f_n}_q \norm{ \left( |h_n|^r-|h|^r \right)}_{q'} < C \ep, \quad \forall \, n \ge n_2. 
\end{align}
The last inequality uses the fact that $(f_n)$ is bounded in $L^q(\Om)$.
From \eqref{cnvg1} and \eqref{cnvg2}, we conclude that $\int_{\Om} f_n |h_n|^r \rightarrow \int_{\Om} f|h|^r.$
\end{proof}

\noi \textbf{Proof of Theorem \ref{optimization}:} By the hypothesis, 
\begin{equation}\label{assumption}
\begin{aligned}
g_0, V_0 \in X:=\left\{ \begin{array}{ll}
L^{\frac{N}{p}}(\Om), &  \text{ if }  N > p; \vspace{0.1 cm} \\
             L^q(\Om); q \in (1, \infty),  &  \text{ if }  N\leq p, 
              \end{array}\right. g_0^+ \not \equiv 0, \text{ and } \norm{V_0^{-}}_{X} \le \displaystyle \frac{1-\de_0}{S^p}. \\
\end{aligned}
\end{equation}
$(i)$ \textbf{Existence of minimizer:} Let $N>p$. Recall that 
\begin{align*}
    \La_{\min}(g_0, V_0)=  \inf \left\{ \La(g,V): g \in \E(g_0), V \in \E(V_0) \right\},
\end{align*}
where $\E(g_0)$  and $\E(V_0)$ are the  set of all rearrangements of $g_0$ and $V_0$ respectively.
Let $(g_n), (V_n)$ be minimizing sequences in $\E(g_0), \E(V_0)$ such that  
\begin{align}\label{min}
    \La_{\min}(g_0, V_0) = \lim_{n \ra \infty} \La(g_n,V_n). 
\end{align}
For brevity, we denote $\La(g_n,V_n)$ as $\La_n$. For each $n \in \N$, using Proposition \ref{rearrangement}-$(i)$, we see that $g_n, V_n$ satisfies all the assumptions as given in \eqref{assumption}. Therefore, applying Theorem \ref{Existence1}, we get
\begin{align}\label{min0}
    \La_n = \frac{\int_{\Om} |\Gr \phi_n|^p + V_n \phi_n^p } { \int_{\Om} g_n \phi_n^p},
\end{align}
where $\phi_n$ is an eigenfunction of \eqref{EVP} corresponding to $\La_n$, $\phi_n>0$ in $\Om$, and $\int_{\Om} g_n \phi_n^p>0$. For $r \in ((p^*)', \frac{N}{p})$, we set $r_1 = \frac{Nr(p-1)}{N-pr}$. Using Proposition \ref{regularity} ($(i)$ of (b)), $(\phi_n) \subset L^{r_1}(\Om)$. It is easy to see that $\Phi_n := \frac{\phi_n}{\norm{\phi_n}_{r_1}}$ is also a positive eigenfunction of \eqref{EVP} corresponding to $\La_n$ normalized as $\norm{\Phi_n}_{r_1} = 1$. Moreover, from \eqref{min} and \eqref{min0},
\begin{align}\label{min1}
    \La_{\min}(g_0, V_0) = \lim_{n \ra \infty} \frac{\int_{\Om} |\Gr \Phi_n|^p + V_n \Phi_n^p} { \int_{\Om} g_n \Phi_n^p}.
\end{align}
Now we show that $(\Phi_n)$ is bounded in $W_{0}^{1, r_2}(\Om)$, where $r_2=\frac{Nr(p-1)}{N-r}>p$. For each $n \in \N$, using Proposition \ref{regularity} ($(ii)$ of (b)), we have the following gradient estimate: 
\begin{align}\label{min2}
    \norm{ \Gr \Phi_n}_{r_2} \le C \norm{(\La_n g_n-V_n)\Phi_n}_r.
\end{align}
We apply the H\"{o}lder's inequality with the conjugate pair $(\frac{N}{pr}, \frac{N}{N-pr})$ to get
\begin{align*}
    \norm{(\La_n g_n-V_n)\Phi_n^{p-1}}^{\frac{1}{p-1}}_r & \le \left( \dis |\La_n g_n-V_n|^{\frac{N}{p}} \right)^{\frac{p}{N(p-1)}} \left( \dis \Phi_n^{r_1} \right)^{\frac{1}{r_1}} \\
   & \le \norm{\La_n g_n - V_n}_{\frac{N}{p}}^{\frac{1}{p-1}} \norm{\Phi_n}_{r_1} = \norm{\La_n g_n - V_n}_{\frac{N}{p}}^{\frac{1}{p-1}}.
\end{align*}
Since $(g_n, V_n) \in \E(g_0) \times \E(V_0)$,  it follows that $\norm{\La_n g_n - V_n}_{\frac{N}{p}} \le \La_n \norm{g_0}_{\frac{N}{p}} + \norm{V_0}_{\frac{N}{p}} \le C$. Therefore, from \eqref{min2}, the sequence $(\norm{ \Gr \Phi_n}_{r_2})$ is bounded. Also, since $r_2<r_1$ and $\Om$ is bounded, we infer that $(\Phi_n)$ is bounded in $L^{r_2}(\Om).$ Thus the sequence $(\Phi_n)$ is bounded in $W_{0}^{1, r_2}(\Om)$. By the reflexivity of $W_{0}^{1, r_2}(\Om)$, there exists a subsequence $(\Phi_{n_k})$ such that $\Phi_{n_k} \rightharpoonup \uphi$ in $W_{0}^{1, r_2}(\Om)$. Since $r_2^*>p^*$, $W_{0}^{1, r_2}(\Om)$ is compactly embedded into $L^{p^*}(\Om)$. Therefore, $\Phi_{n_k} \ra \uphi$ in $L^{p^*}(\Om)$ and $\uphi \ge 0$  in $\Om$. Further, the sequences $(g_{n_k})$ and $(V_{n_k})$ are bounded in $L^{\frac{N}{p}}(\Om)$. By the reflexivity of $L^{\frac{N}{p}}(\Om)$, up to a subsequence $g_{n_k} \rightharpoonup \tilde{g}$ and $ V_{n_k} \rightharpoonup \tilde{V}$ in $L^{\frac{N}{p}}(\Om)$. Hence using Lemma \ref{convergence}, we get
\begin{align*}
    \lim_{k \rightarrow \infty} \dis g_{n_k} \Phi_{n_k}^p = \dis \tilde{g} (\uphi)^p \;  \text{ and } \; \lim_{k \rightarrow \infty} \dis V_{n_k} \Phi_{n_k}^p = \dis \tilde{V} (\uphi)^p.
\end{align*}
Therefore, \eqref{min1} and the weak lower semicontinuity of $\norm{\Gr( \cdot)}_p$ yield $$\La_{\min}(g_0, V_0) \ge \frac{\int_{\Om} |\Gr \uphi|^p + \tilde{V} (\uphi)^p} { \int_{\Om}  \tilde{g} (\uphi)^p}, \text{ where } (\tilde{g}, \tilde{V}) \in \overline{\E(g_0)} \times \overline{\E(V_0)}.$$
Furthermore, from Proposition \ref{rearrangement}-$(iii)$ there exists $(\ug ,\uv) \in \E(g_0) \times \E(V_0)$ such that  $\int_{\Om} \uv (\uphi)^p \le \int_{\Om} \tilde{V} (\uphi)^p$ and $\int_{\Om} \ug (\uphi)^p \ge \int_{\Om} \tilde{g} (\uphi)^p$. Using these inequalities it follows that
\begin{align*}
    \La_{\min}(g_0, V_0) \ge \frac{\int_{\Om} |\Gr \uphi|^p + \tilde{V} (\uphi)^p} { \int_{\Om}  \tilde{g} (\uphi)^p} \ge \frac{\int_{\Om} |\Gr \uphi|^p + \uv (\uphi)^p} { \int_{\Om} \ug (\uphi)^p} \ge \La_{\min}(g_0, V_0).
\end{align*}
Thus $\La_{\min}(g_0, V_0)$ is attained at $(\ug,\uv) \in \E(g_0) \times \E(V_0)$. For $N \le p$, the existence of minimizer follows from \cite[Theorem 3.4]{Leandro}.

$(ii)$ \textbf{Existence of maximizer:} 
Recall that
$$\La_{\max}(g_0,V_0) :=  \sup \left\{ \La(g,V): (g, V) \in \E(g_0) \times \E(V_0) \right\}.$$
Let $(g_n) \subset \E(g_0)$ and $(V_n) \subset \E(V_0)$ be maximizing sequences, i.e.,
\begin{align}\label{max}
      \La_{\max}(g_0,V_0) = \lim_{n \ra \infty} \La(g_n,V_n) = \lim_{n \ra \infty} \frac{\int_{\Om} |\Gr \phi_n|^p + V_n \phi_n^p} { \int_{\Om} g_n \phi_n^p},
\end{align}
where $\phi_n$ is a positive eigenfunction of \eqref{EVP} corresponding to $\La(g_n,V_n)$ (by Proposition \ref{rearrangement}-$(i)$ and Theorem \ref{Existence1}). As before, we denote $\La(g_n,V_n)$ as $\La_n$. Since the sequences $(g_n)$ and $(V_n)$ are bounded in $X$, by the reflexivity of $X$, up to a subsequence $g_n \rightharpoonup g$ and $V_n \rightharpoonup V$ in $X$. Now $\int_{\Om} g_n f \rightarrow \int_{\Om} gf, \; \forall \, f \in X'$, where $X'$ is the dual of $X$. Further, since $g_n \in \E(g_0)$ and $g_0 \geq 0$, it follows from Proposition \ref{rearrangement}-$(ii)$ that $\int_{\Om} g_n = \int_{\Om} g_0.$  Now, by taking $f=1$, we obtain 
\begin{align*}
    \int_{\Om} g = \lim_{n \ra \infty} \int_{\Om} g_n = \int_{\Om} g_0 > 0.
\end{align*}
Therefore, $g^+ \not \equiv 0$ on a set of positive measure. Also, from the weak lower semicontinuity of $\norm{\cdot}_{X}$, 
\begin{align*}
    \norm{V^-}_{X} \le \norm{V}_{X} \le \liminf_{n \ra \infty} \, \norm{V_n}_{X} \le \frac{1-\de_0}{S^p}.
\end{align*}
Thus $g,V$ satisfies all the assumptions in \eqref{assumption}, and by Theorem \ref{Existence1}, there exists an eigenfunction $\ophi$ of \eqref{EVP} corresponding to $\La(g,V)$. Now we write
\begin{align}\label{max1}
      \La_n = \frac{\int_{\Om} |\Gr \phi_n|^p + V_n \phi_n^p } { \int_{\Om} g_n \phi_n^p}  \le \frac{\int_{\Om} |\Gr \ophi|^p + V_n (\ophi)^p} { \int_{\Om} g_n (\ophi)^p} 
       = \La(g,V) \frac{\int_{\Om} g (\ophi)^p}{\int_{\Om} g_n (\ophi)^p} + \frac{\int_{\Om} (V_n-V) (\ophi)^p } { \int_{\Om} g_n (\ophi)^p}.
\end{align}
From the Sobolev embedding $\wp \hookrightarrow Y$, where $Y=L^{p^*}(\Om)$ (if $N>p$) and $Y = L^{pq'}(\Om)$ (if $N \le p$), we have $(\ophi)^p \in X'$. Therefore, $\int_{\Om} (V_n-V) (\ophi)^p \ra 0$ and $\int_{\Om} (g_n-g) (\ophi)^p \ra 0$, as $n \ra \infty$. Now using \eqref{max1}, we obtain $$\limsup_{n \ra \infty} \La_n \le \La(g,V).$$ Therefore, $ \La_{\max}(g_0,V_0) \leq \La(g,V),$ where $g \in \overline{\E(g_0)}$ and $V \in \overline{\E(V_0)}$. Further, from the rearrangement inequality (Proposition \ref{rearrangement}-$(iii)$) there exists $(\og ,\ov) \in \E(g_0) \times \E(V_0)$ such that  $\int_{\Om} \ov (\ophi)^p \ge \int_{\Om} V (\ophi)^p$ and $\int_{\Om} \og (\ophi)^p \le \int_{\Om} g (\ophi)^p$. Therefore, 
\begin{align*}
      \La_{\max}(g_0,V_0) \le \frac{\int_{\Om} |\Gr \ophi|^p + V (\ophi)^p } { \int_{\Om} g (\ophi)^p} \le \frac{\int_{\Om} |\Gr \ophi|^p + \ov (\ophi)^p } { \int_{\Om} \og (\ophi)^p} \le  \La_{\max}(g_0,V_0).
\end{align*}
Thus $\La_{\max}(g_0,V_0)$ is attained at $(\og, \ov) \in \E(g_0) \times \E(V_0)$. 
\qed

\begin{remark}
Notice that, in order to get the existence of optimizers, we have used the relexivity of the space $X$. Naturally, when $X = L^1(\Om)$, the above procedure fails due to the lack of reflexivity.
\end{remark}

In the following proposition, we give a characterization of minimizers $\ug$ and $\uv$.
\begin{proposition}\label{characterization}
Let $g_0, V_0$ be as given in Theorem \ref{optimization}. Let $(\uphi, \ug, \uv)$ be an optimal triple. Then there exists an increasing function $F : \R \mapsto \R$ and a decreasing function $G : \R \mapsto \R$ such that 
\begin{align*}
    \ug= F \circ \uphi \text{  and  }  \uv = G \circ \uphi \; \text{ in} \; \Om.
\end{align*}
\end{proposition}

\begin{proof}
Proof follows using Theorem \ref{optimization}-$(i)$ and the similar set of arguments as given in \cite[Theorem 3.5]{Leandro}.
\end{proof}

Next, we study the uniqueness for the maximization problem \eqref{OurProb_max}. In \cite[Theorem 4.4]{Cuccu}, authors proved the uniqueness of maximizer when $\Om$ is a ball and $g_0\in L^\infty(\Om)$ is nonnegative. Here we establish the uniqueness for more general domains $\Om$ and nonnegative $g_0\in X$ ($X$ is as in \eqref{A}) by extending ideas of the preceding paper. In order to get this, we derive the weak continuity of the map $g\mapsto \La(g,0)$ in $X$. For brevity, we denote $\La(f)=\La(f,0)$ for a function $f$.

\begin{proposition}
Let $g_0$ satisfies \eqref{A}, $g_0 \ge 0$ and $V_0 = 0$. Then the following holds:
\begin{enumerate}[(i)]
    \item Let $\overline{\E(g_0)}$ be the weak closure of $\E(g_0)$ in $X$. Then the map $g \mapsto \La(g)$ is continuous on $\overline{\E(g_0)}$, i.e., for every sequence $(g_n)$ in $\overline{\E(g_0)}$ if $g_n \rightharpoonup g$ in $X$, then $\La(g_n) \ra \La(g)$,
    \item There exists a unique maximizer of $\La_{\max}(g_0,0)$.
\end{enumerate}
\end{proposition}
\begin{proof}
$(i)$ Let $N>p$ and $g_0 \in L^{\frac{N}{p}}(\Om)$. Let $(g_n)$ be a sequence in $\overline{\E(g_0)}$ such that $g_n \wra g$ in $L^{\frac{N}{p}}(\Om)$. We show that $\La(g_n) \ra \La(g)$. For each $n \in \N$, since $g_n \in \overline{\E(g_0)}$, there exists a sequence $(g_{n,m})$ in $\E(g_0)$ such that $g_{n,m} \rightharpoonup g_n$ in $L^{\frac{N}{p}}(\Om)$. Now for every $f \in (L^{\frac{N}{p}}(\Om))'$, we have $\int_{\Om} g_{n,m} f \rightarrow \int_{\Om} g_nf$, as $m \ra \infty$  and $\int_{\Om} g_n f \rightarrow \int_{\Om} gf,$ as $n \ra \infty$. In particular, for $f=1$,
\begin{align*}
  \dis g_n = \lim_{m \ra \infty} \int_{\Om} g_{n,m} =  \int_{\Om} g_0 \; \text{  and  } \; \dis g = \lim_{n \ra \infty} \int_{\Om} g_n = \int_{\Om} g_0.
\end{align*}
Therefore, for each $n \in \N$, $\text{supp}(g_n^+)$ and $\text{supp}(g^+)$ have positive measure. Hence using Theorem \ref{Existence1}, there exist positive eigenfunctions $\phi_n$ and $\phi $ of \eqref{EVP} corresponding to $\La(g_n)$ and $\La(g)$, respectively. Further,
\begin{align*}
      \La(g_n) = \frac{\int_{\Om} |\Gr \phi_n|^p} { \int_{\Om} g_n \phi_n^p}  \le \frac{\int_{\Om} |\Gr \phi|^p} { \int_{\Om} g_n \phi^p}  = \La(g) \frac{\int_{\Om} g \phi^p}{\int_{\Om} g_n \phi^p}.
\end{align*}
This yields $\underset{n \ra \infty}{\limsup} \, \La(g_n) \le \La(g)$. On the other hand, following the steps as given in the proof of Theorem \ref{optimization}-$(i)$, we get a sequence $(\Phi_n)$ of eigenfunctions of \eqref{EVP} such that 
 \begin{align}\label{max3}
     \La(g_n) = \frac{\int_{\Om} |\Gr \Phi_n|^p} { \int_{\Om} g_n \Phi_n^p}, \Phi_n \rightharpoonup \uphi \text{ in } \wp, \text{ and } \int_{\Om} g_{n} \Phi_n^p \rightarrow \int_{\Om} g (\uphi)^p.
 \end{align}
Hence \eqref{max3} and the weak lower semicontinuity of $\norm{\Gr( \cdot)}_p$ give
 \begin{align*}
     \liminf_{n \ra \infty} \La(g_n) \ge  \frac{\int_{\Om} |\Gr \uphi|^p} { \int_{\Om}  g (\uphi)^p} \ge \La(g).
 \end{align*}
Thus the sequence $(\La(g_n))$ converges to $\La(g)$. For $N \le p$, proof follows using the similar set of arguments.

$(ii)$ We consider the following maximization problem: 
\begin{align}\label{maxproblem}
    \overline{\La}_{\max}(g_0) = \sup \left\{ \La(g): g \in \overline{\E(g_0)} \right\}.
\end{align} 
\noi \underline{Step 1:} First, we show that the maximizer of \eqref{maxproblem} is attained in $\E(g_0)$. Let $(g_n)$ be a maximizing sequence in $\overline{\E(g_0)}$ such that $\La(g_n) \rightarrow \overline{\La}_{\max}(g_0)$. Since the set  $\overline{\E(g_0)}$ is weakly sequentially compact (by \cite[Lemma 2.2]{Burton}), up to a subsequence $g_n \rightarrow \tilde{g}$ in $\overline{\E(g_0)}$ (i.e., $g_n \wra \tilde{g}$ in $X$). Using the continuity of $g \mapsto \La(g)$, we have $\La(\g) = \underset{n \rightarrow \infty}{\lim} \La(g_n) = \overline{\La}_{\max}(g_0)$. Further, using Proposition \ref{rearrangement}-$(iii)$, there exists $\hat{g} \in \E(g_0)$ such that $\La(\g) \le \La(\hat{g})$. Thus, $\La(\hat{g}) = \overline{\La}_{\max}(g_0)$.   
 
\noi \underline{Step 2:} Next, we claim that the maximizer $\hat{g}$ of \eqref{maxproblem} is unique. One can verify that 
$$\hat{\La}_{\min}(g_0) := \inf \left\{ \frac{1}{\La(g)}: g \in \overline{\E(g_0)} \right\} = \left( \overline{\La}_{\max}(g_0) \right)^{-1}.$$ 
Thus the uniqueness of  maximizer for $\overline{\La}_{\max}(g_0)$ is equivalent to the uniqueness of minimizer for $ \hat{\La}_{\min}(g_0)$. Suppose there exists $g_1, g_2 \in \E(g_0)$ such that $\frac{1}{\La(g_1)} = \frac{1}{\La(g_2)} = \hat{\La}_{\min}(g_0)$. For $t \in (0,1)$, set $f_t = tg_1 + (1-t)g_2$. Since $\overline{\E(g_0)}$ is convex (by \cite[Lemma 2.2]{Burton}), $f_t \in \overline{\E(g_0)}$. Let $\phi_{f_t}, \phi_{g_1},$ and $\phi_{g_2}$ be eigenfunctions of \eqref{EVP} corresponding to $\La(f_t), \La(g_1),$ and $\La(g_2)$. Then 
\begin{align*}
  \hat{\La}_{\min}(g_0) \le \frac{1}{\La(f_t)} = t \frac{\int_{\Om} g_1 \phi_{f_t}^p}{ \int_{\Om} |\Gr \phi_{f_t}|^p} + (1-t) \frac{\int_{\Om} g_2 \phi_{f_t}^p}{ \int_{\Om} |\Gr \phi_{f_t}|^p} & \le t \frac{\int_{\Om} g_1 \phi_{g_1}^p}{ \int_{\Om} |\Gr \phi_{g_1}|^p} + (1-t) \frac{\int_{\Om} g_2 \phi_{g_2}^p}{ \int_{\Om} |\Gr \phi_{g_2}|^p} \\
  &= t \frac{1}{\La(g_1)} + (1-t)\frac{1}{\La(g_2)} = \hat{\La}_{\min}(g_0).
\end{align*}
Hence the equality holds in each of the above inequalities. Therefore, the following equations hold weakly: 
$$-\De_p \phi_{f_t} = \La(g_1)g_1\phi_{f_t}^{p-1}, \text{ and }-\De_p \phi_{f_t} = \La(g_2)g_2\phi_{f_t}^{p-1} \, \text{ in }  \Om.$$ 
From the above identities, it follows that $\La(g_1) g_1 = \La(g_2) g_2$ in $\Om$. Further, $g_0 \in L^1(\Om)$ and using Proposition \ref{rearrangement}-$(ii)$ we have $\int_{\Om} g_1 = \int_{\Om} g_2 = \int_{\Om} g_0 >0$. Therefore, $\La(g_1) = \La(g_2)$ and $g_1 = g_2$ in $\Om$. Thus the minimizer of $\hat{\La}_{\min}(g_0)$ is unique, and the uniqueness of $\hat{g}$ follows immediately.

\noi \underline{Step 3:} From Step 1, we have $$\La_{\max}(g_0)  \le \overline{\La}_{\max}(g_0) = \La(\hat{g}) \le \La_{\max}(g_0).$$
Therefore, using Step 2, it is evident that the maximizer of \eqref{OurProb_max} is unique. 
\end{proof}
\begin{remark}
In general, minimizer of \eqref{OurProb_min} need not be unique (see Remark \ref{nonuniq}). However, when $\Om$ is a ball, there
exists a unique minimizer for \eqref{OurProb_min}; cf. \cite[Theorem 3.3]{Cuccu}.
\end{remark}

\section{Symmetry of minimizers}\label{symmetry_section}
This section is devoted to studying the various symmetry of the minimizers of \eqref{OurProb_min}. First, we state a strong maximum principle due to Brezis and Ponce in \cite[Corollary 4]{Brezis-Ponce}, which will be used in our proof of Theorem \ref{min_thm}.

\begin{proposition}[Strong Maximum Principle]\label{STM}
 Let $O \subset \RN$ be a bounded domain and $V \in L^1_{loc}(O)$ with $V \geq 0$ a.e. in $O$. Assume that $\phi\geq 0$, $V\phi \in L^1_{loc}(O)$ and $\De \phi$ is a Radon measure on $O$. Suppose that the following inequality holds in the sense of distribution:
\begin{align*}
    -\De \phi + V \phi \ge 0.
\end{align*}
Then either $\phi \equiv 0$ or $\phi > 0$ a.e. in $O$.
\end{proposition}

For the rest of this section, we denote $\uphi$ as $\phi.$ \vspace{0.2 cm} \\
\noi \textbf{Proof of Theorem \ref{min_thm}}: $(i)$ Let $H\in\H_0$. By the hypothesis, $\Om = \Om_H, \sigma_H(\Om)\neq \Om, V_0=0$, $g_0$ satisfies \eqref{A} with $  g_0 \ge 0$, and $\ug$ is given in Theorem \ref{optimization}-$(i)$. For simplicity, we set $\La_{\min}(g_0):=\La_{\min}(g_0,0)$.  
From Theorem \ref{optimization}-$(i)$ and Theorem \ref{Existence1}, there exists $\phi \in \hh2$ such that $\phi>0$  in $\Om$ and 
\begin{align} \label{attain_1}
 \displaystyle \La_{\min}(g_0)  =\La(\underline{g})= \frac{\int_{\Om}  |\nabla \phi|^2} {\int_{\Om} \underline{g} \phi^2}.
\end{align}
Using Proposition \ref{pola_bounded}-$(i)$, we see that $\ug_H \in \E(g_0)$. Hence, $\ug_H \ge 0$ and $\ug_H$ satisfies \eqref{A}. Thus, using Theorem \ref{Existence1}, we infer that $\La(\ug_H)$ is achieved. Further, since $\ug_H \in \E(g_0)$, it follows that 
\begin{align} \label{attain_2}
    \La(\underline{g}_H) \geq \La_{\min}(g_0).
\end{align}
Now from the Hardy-Littlewood inequality (Proposition \ref{Hardy_Littlewood}-$(i)$),
\begin{align}\label{Hardy_1}
    \int_{\Om} \underline{g} \phi^2 \le \int_{\Om} \ug_H (\phi_H)^2 ,
\end{align}
where we also used the fact that $ (\phi^2)_H=(\phi_H)^2$ (as $\phi>0$). Furthermore, since $\phi\in \hh2$ and $\phi>0$  in $\Om$, by Proposition \ref{pola_bounded}-$(ii)$,  we have $\phi_H \in \hh2$ and $\norm{\Gr \phi}_2 = \norm{\Gr \phi_H}_2$. Therefore, using \eqref{attain_1},  \eqref{attain_2}, and \eqref{Hardy_1}, we get
\begin{align*}
\La(\ug) = \frac{\int_{\Om}  |\nabla \phi|^2 }{\int_{\Om} \ug \phi^2} 
 \ge \frac{\int_{\Om}  |\Gr \phi_H|^2}{\int_{\Om} \ug_H (\phi_H)^2}
\ge \La(\ug_H) \ge \La(\ug).
\end{align*}
Thus the equality occurs in each of the above inequalities. As a consequence, $\phi$  and $\phi_H$  satisfy the following equations weakly: 
\begin{equation}\label{Eq_H1}
   -\De \phi = \La(\ug) \ug \phi  \text{ in }  \Om, \text{ and }  -\De \phi_H = \La(\ug) \ug_H \phi_H  \text{ in }  \Om.
\end{equation}
Set $w=\phi_H-\phi$. Then $w \ge 0$  in $\Om \cap H$, and from \eqref{Eq_H1}, $w$ satisfies the following equation weakly:
\begin{equation}\label{Diff_eq1}
    \begin{aligned}
  -\De w = \La(\ug) (\ug_H\phi_H-\ug \phi) \text{ in }  \Om\cap H, \; w = 0  \text{ on } \pa(\Om\cap H).
   \end{aligned}
\end{equation}
Moreover, since $g_0\geq 0$, we get $\ug \ge 0$ and hence $\ug_H\phi_H-\ug \phi\geq 0$  in $\Om\cap H$. Therefore, applying the strong maximum principle (Proposition \ref{STM}) and using \eqref{Diff_eq1} we obtain $w>0$ or $w = 0$  in $\Om\cap H$, i.e.,  
\begin{equation}\label{Max_phi}
     \phi_H>\phi \text{ in } \Om\cap H, \; \text{ unless }\; \phi_H=\phi \text{ in }  \Om\cap H.
\end{equation}
Further, since $\sigma_H(\Om)\neq \Om$, using  Proposition \ref{fact_pola_domain}-$(iii)$, there exists $A \subset \Om\cap H$ such that $|A|>0$ and $\sigma_H(A)\subset \Om^c\cap \overline{H}^c$. For $x \in A$, from Definition \ref{pola_def}-(ii), 
$\phi_H(x) = \tilde{\phi}_H(x) = \tilde{\phi}(x) = \phi(x),$ 
i.e., $\phi_H=\phi$  in $A$. Therefore, from \eqref{Max_phi}, we must have $\phi_H=\phi$  in $\Om\cap H$, i.e., $\phi\geq \phi\circ\sigma_H$  in $\Om\cap H$. Consequently, we get $\phi_H=\phi$  in $\Om$. Moreover, from \eqref{Diff_eq1}  the conclusion $\ug_H=\ug$  in $\Om$ follows immediately. 
 
$(ii)$ Let $H\in\H_0$ be such that $\sigma_H(\Om) = \Om$. By the hypothesis, $g_0, V_0$ satisfy \eqref{A}, with $g_0, V_0 \ge 0$, and $\ug, \uv$ are given in Theorem \ref{optimization}-$(i)$. Using  Theorem \ref{optimization}-$(i)$ and Theorem \ref{Existence1}, there exists positive $\phi \in \hh2$ such that
\begin{align} \label{attain_3}
 \displaystyle \La_{\min}(g_0, V_0)  =\La(\ug, \uv)= \frac{\int_{\Om}  |\nabla \phi|^2 + \uv \phi^2} {\int_{\Om} \underline{g} \phi^2}.
\end{align}
From Proposition \ref{pola_bounded}-$(i)$ and Remark \ref{Facts_polarization}-$(iii)$, we obtain $(\ug_H, \uv^H) \in \E(g_0)\times \E(V_0)$. Hence, $\ug_H, \uv^H \ge 0$,  and  $\ug_H, \uv^H $ satisfy \eqref{A}. Therefore, by Theorem \ref{Existence1}, $\La(\ug_H, \uv^H)$ is achieved. Further, from the Hardy-Littlewood inequality (Proposition \ref{Hardy_Littlewood}-$(i)$), reverse Hardy-Littlewood inequality (Proposition \ref{Hardy_Littlewood}-$(ii)$) and using Proposition \ref{pola_bounded}-$(ii)$, we obtain
\begin{align*}
    \int_{\Om} \underline{g} \phi^2 \le \int_{\Om} \ug_H (\phi_H)^2 , \quad \dis \uv \phi^2 \ge \dis \uv^H \phi_H^2, \quad \norm{\Gr \phi}_2 = \norm{\Gr \phi_H}_2.
\end{align*}
Therefore, \eqref{attain_3} yields
\begin{align*}
    \La_{\min}(g_0,V_0) = \frac{\int_{\Om}  |\nabla \phi|^2 + \uv \phi^2} {\int_{\Om} \ug \phi^2} \ge \frac{\int_{\Om}  |\nabla \phi_H|^2 + \uv^H (\phi_H)^2} {\int_{\Om} \ug_H (\phi_H)^2} \ge \La(\ug_H, \uv^H) \ge \La_{\min}(g_0,V_0).
\end{align*}
Since equality occurs in each of the above inequalities, the following equations hold weakly:
\begin{equation}\label{Eq_H2}
  -\De \phi + \uv \phi = \La(\ug, \uv) \ug \phi, \, \text{ and } -\De \phi_H + \uv^H \phi_H = \La(\ug, \uv) \ug_H \phi_H  \text{ in }  \Om.
\end{equation}
As before, we set $w = \phi_H - \phi$ and using \eqref{Eq_H2} see that $w \in \hh2$ satisfies the following equation weakly:
\begin{equation}\label{Diff_eq2}
  -\De w + \uv w \ge -\De w + (\uv^H\phi_H - \uv \phi) \ge 0  \text{ in } \Om\cap H, \; w = 0 \text{ on } \pa(\Om\cap H).
\end{equation}
Further, $\int_{\Om \cap H} \uv w  \le \left( \int_{\Om \cap H} \uv \right)^{\frac{1}{2}}  \left( \int_{\Om \cap H} \uv w^2 \right)^{\frac{1}{2}} < \infty.$
Therefore, by Proposition \ref{STM}, we conclude that either $\phi_H>\phi$ or $\phi_H=\phi$  in  $\Om\cap H.$ Now we consider these two possibilities separately:

 $(a)$ Let $\phi_H=\phi$  in  $\Om\cap H$, i.e.,  $\phi \circ \sigma_H \le \phi$ in $\Om \cap H$. Then using Proposition \ref{characterization}, we get \begin{align*}
     &\ug \circ \sigma_H = F \circ (\phi \circ \sigma_H) \le F \circ \phi = \ug \text{ in }  \Om \cap H. \\
     & \uv \circ \sigma_H = G \circ (\phi \circ \sigma_H) \ge G \circ \phi = \uv \text{ in }  \Om \cap H.
 \end{align*}
 Therefore,  $ \phi_H=\phi$, $\ug_H=\ug,$ and $ \uv^H=\uv$ in $\Om$. 
 
 $(b)$ If $\phi_H > \phi$  in  $\Om\cap H$, i.e., $\phi \circ \sigma_H>\phi$  in $\Om\cap H$, then using Proposition \ref{characterization}, $\ug \circ \sigma_H > \ug$ and $\uv \circ \sigma_H < \uv$  in $\Om \cap H$. Therefore, we get $\phi^H=\phi, \ug^H=\ug,$ and $\uv^H = \uv \circ \sigma_H$  in $\Om$. Further, using Definition \ref{pola_def}-(ii), it follows that $\uv_H=\uv$ in $\Om$. \\
 Combining both possibilities, we complete the proof.
\qed
 
\begin{remark}[Radiality on ball] \label{Cuccus_rslt}
If $\Om=B_1(0)$, then we  have $\Om_H=\Om$ and $\sigma_H(\Om)\neq \Om$ for every $H\in \H(0)$. Therefore, by Theorem \ref{min_thm}-$(i)$,  $\phi_H=\phi$ and $\ug_H=\ug$ in $\Om$ for all $H\in \H(0)$. Hence from Proposition \ref{Schwarz_char}, we conclude that $\phi$ and $\ug$ are radial and radially decreasing on $\Om$. For $1<p <\infty$, this result has been proved in \cite{Cuccu} with $V_0=0$ and in \cite{Prajapat} with $V_0\geq 0$. Here we recover  the same result for $p=2$ and $V_0=0$, using the polarization invariance structure of a minimizing weight and the associated first eigenfunctions.
\end{remark}

\begin{proposition}
Let $\Om$ be a bounded domain containing 0, $g_0=\chi_E$, where $E\subsetneq\Om$ with $0<|E|<|\Om|$ and $V_0=0$. If $\Om=B_R(0)$ for some $R>0$, then $\ug=\chi_{B_r(0)}$ for some $r>0$ such that $|E|=|B_r(0)|$. Furthermore, the converse is also true.
\end{proposition}
\begin{proof}
Let $\Om=B_R(0)$. Since $g_0=\chi_E$, we have $\ug=\chi_F$ for some $F\subsetneq \Om$ with $|F|=|E|$. Now by Remark \ref{Cuccus_rslt}, $\ug$ is radial and radially decreasing. Thus, $\ug(0)=1$. Moreover, $|\{\ug = 1\}|=|\{g_0 =1\}|=|E|.$ Thus $F$ must be a ball centered at origin, i.e., $F=B_r(0)$ for some $r>0$ such that $|E|=|B_r(0)|$.
The proof of the converse result follows adapting the similar ideas used in \cite[Theorem 2]{Lamboley}.
\end{proof}

If a domain is symmetric with respect to the hyperplane $\pa H$, where $H\in \H_0$, then Theorem \ref{min_thm}-$(ii)$ states that any optimal triple remains either polarization invariant or dual-polarization invariant with respect to $H$. This is the finest result (in a certain sense) one can expect without any further assumptions on the domain. The following remark emphasizes this assertion.

\begin{remark}\label{non_rad} (Nonradiality on concentric annulus)
Let $\Om=B_{R+1}(0)\setminus \overline{B_R(0)}\subset \R^2$, where $R>0$ is sufficiently large and $V_0=0$. Then proceeding in the same way as in the proof of \cite[Theorem 6]{Chanillo}, we can show that there exists $g_0=\chi_D$, where $D\subset \Om$, such that $\ug$ is not radial. Consequently, any first eigenfunction $\phi$ associated to $\ug$ is  nonradial. 
\end{remark}
\begin{remark}[Nonuniqueness of minimizer]\label{nonuniq}
Nonuniqueness of minimizer of \eqref{OurProb_min} follows from the asymmetric nature of minimizers, as mentioned in Remark \ref{non_rad}. More precisely, let us choose a concentric annulus $\Om$ centered at the origin and a weight function $g_0$ such that a minimizing weight $\ug$ is not radial. Thus there exists $H\in \H_0$ such that $\ug\neq \ug_H$. Further, we show that $\ug_H$ is also a minimizer for \eqref{OurProb_min} (in the proof of Theorem \ref{min_thm}-$(i)$). Hence the minimizer for \eqref{OurProb_min} is not unique.
\end{remark}
 
As a consequence of Theorem \ref{min_thm}-$(i)$, next, we prove Corollary \ref{Steiner_theo}, which assures that an optimal pair $(\phi,\ug)$ preserves the Steiner symmetry if the underlying domain is Steiner symmetric. \vspace{0.2 cm} \\
\noi \textbf{Proof of Corollary \ref{Steiner_theo}:}
Let $H\in \H_0$ and $\Om$ be Steiner symmetric with respect to $\pa H$. Since the Laplace operator is invariant under isometries, without loss of generality, we assume that $H=\{\big(x_1,x_2,\dots,x_N\big)\in \R^N: x_N<0\}$, i.e., $\Om$ is Steiner symmetric with respect to the hyperplane $\pa H=\{\big(x_1,x_2,\dots,x_N\big)\in \R^N: x_N=0\}$. Let $\H_*\subset \H_0$ be the collection of all open half-spaces $\widetilde{H}$ containing $\pa H$ such that $\pa \widetilde{H}$ is parallel to $\pa H$. Therefore, using Proposition \ref{Steiner_char}-$(i)$, we have $\Om_{\widetilde{H}}=\Om$ for all $\widetilde{H}\in \H_*$. Since $\Om$ is symmetric with respect to $\pa H$, it is easy to observe that 
$\sigma_{\widetilde{H}}(\Om)\neq \Om$ for every $\widetilde{H}\in \H_*$. Hence by Theorem \ref{min_thm}-$(i)$, we get $\phi_{\widetilde{H}}=\phi$ and $\ug_{\widetilde{H}}=\ug$ in $\Om$ for all $\widetilde{H}\in \H_*$. Therefore by Proposition \ref{Steiner_char}-$(ii)$, we conclude that $\phi$ and $\ug$ are Steiner symmetric in $\Om$. 
\qed

Now we study the foliated Schwarz symmetry of the minimizers. First, we prove Theorem \ref{Foliation_minimization}. Then we discuss some of its consequences. \vspace{0.2 cm} \\
\noi \textbf{Proof of Theorem \ref{Foliation_minimization}:}
$(i)$ By the hypothesis, $\Om=B_R(0)\setminus \overline{B_r(0)}$, where $0<r<R$. Recall that $\widehat{\H}_0=\{H\in \H_0:0\in \pa H\}$. For each $H\in \widehat{\H}_0$, we have $\sigma_H(\Om)=\Om$, and we apply Theorem \ref{min_thm}-$(ii)$ to get $\phi_H=\phi$ or $\phi^H=\phi$  in $\Om$. Therefore, from Proposition \ref{Foliated_char}-$(i)$, there exists $\ga\in \S^{N-1}$ such that $\phi$ is foliated Schwarz symmetric in $\Om$ with respect to $\ga$. Hence using Proposition \ref{Foliated_char}-$(ii)$, we get $\phi_H=\phi,\;\forall\,H\in \widehat{\H}_0(\ga)$. Further, following the arguments as given in the proof of Theorem \ref{min_thm}-$(ii)$, we also get $\ug_H=\ug$ and $ \uv^H=\uv,\;\forall\,H\in \widehat{\H}_0(\ga)$. Therefore, from the sufficient condition for the foliated Schwarz symmetrization (Proposition \ref{Foliated_char}-$(ii)$), we conclude $\ug$ is foliated Schwarz symmetric in $\Om$ with respect to $\ga$. Moreover, since $\uv^H=\uv$ for $H\in \widehat{\H}_0(\ga)$, from  Remark \ref{Facts_polarization}-$(ii)$ we have $\uv_{\widetilde{H}}=\uv$, where $\widetilde{H}=\overline{H}^c\in \widehat{\H}_0(-\ga)$. Since $H$ is arbitrary, $\uv_H=\uv,\;\forall\,H\in \widehat{\H}_0(-\ga)$. Now again from Proposition \ref{Foliated_char}-$(i)$, it follows that $\uv$ is foliated Schwarz symmetric in $\Om$ with respect to $-\ga$.

\noi$(ii)$ In this case, $\Om=B_R(0)\setminus \overline{B_r(te_1)}$, where $0< t<R-r$, and $V_0=0$. Recall that
$\widehat{\H}_0(-e_1)=\{H\in \widehat{\H}_0:-e_1\in H\}.$
 It is easy to observe that $\Om_H=\Om$ and $\sigma_H(\Om)\neq \Om$ for every $H\in \widehat{\H}_0(-e_1)$. Thus by Theorem \ref{min_thm}-$(i)$, we have $\phi_H=\phi$ and $\ug_H=\ug$ for every $H\in \widehat{\H}_0(-e_1)$. Therefore, $\phi$ and $\ug$ are foliated Schwarz symmetric with respect to $-e_1$ in $\Om$ (by Remark \ref{Foliated_noncon}). 
\qed

\begin{corollary}
Let $u_1$ be a positive eigenfunction associated to the first eigenvalue $\la_1$ of the following eigenvalue problem on $\Om=B_R(0)\setminus \overline{B_r(te_1)}$, where $0<t<R-r$:
\begin{equation*}
    -\De u=\la u \;\text{in}\;\Om\; ;\;u=0\;\text{on}\; \pa \Om.
\end{equation*}
Then $u_1$ is foliated Schwarz symmetric with respect to $-e_1$ on $\Om$.
\end{corollary}
\begin{proof}
We note that if $g_0=1$, then $\E(g_0)=\{g_0\}$. Thus $\La_{\min}(g_0)=\la_1$ and hence $(u_1,g_0)$ is an optimal pair. Now the assertion follows from $(ii)$ of Theorem \ref{Foliation_minimization}.
\end{proof}

\begin{remark}\label{rema_foliation}
\begin{enumerate}[(i)]
    \item  Let $\Om=B_R(0)\setminus\overline{B_r(te_1)},\;0< t<R-r$ and $V_0=0$. Let $(\phi, \ug)$ be an optimal pair as given in Theorem \ref{optimization}-$(i)$. Then from Theorem \ref{Foliation_minimization}-$(ii)$, $\phi$ is axially symmetric with respect to the axis $\R e_1$ and decreasing in the polar angle $\arccos\big( \frac{-x\cdot e_1}{|x|}\big)$. If $g_0\in L^q(\Om)$ with $q>\frac{N}{2}$, then continuity of $\phi$ (Proposition \ref{regularity}) along with the foliated Schwarz symmetry ensures that maxima of $\phi$ is attained on $ \Om\cap (-\R^+ e_1)$.

\item Let $\Om=B_R(0)\setminus \overline{B_r(0)}$, where $0<r<R$. Also let $g_0=1$, $V_0=\al\chi_D$, where $\al>0$, and  $D\subset \Om$. Observe that in this case $\uv=\al \chi_E$ for some $E\subset \Om$ with $|E|=|D|$. In \cite[Theorem 6]{Chanillo}, the authors showed that there exist $R,r,V_0$ for which $E$ is not rotationally symmetric. However, using Theorem \ref{Foliation_minimization}-$(i)$ we conclude that for any $ R,r$ with $ 0<r<R$ and $D$, the function $\al \chi_E$ and hence $E$ is axially symmetric with respect to some axis passing through the origin. Thus the axial symmetry of $E$ does not depend on the choices of $R$, $r,$ and $D$.
\end{enumerate}
\end{remark}
\begin{figure}
    \includegraphics[width=\textwidth,trim={0 1.2cm 0 1.6cm},clip,scale=0.6]{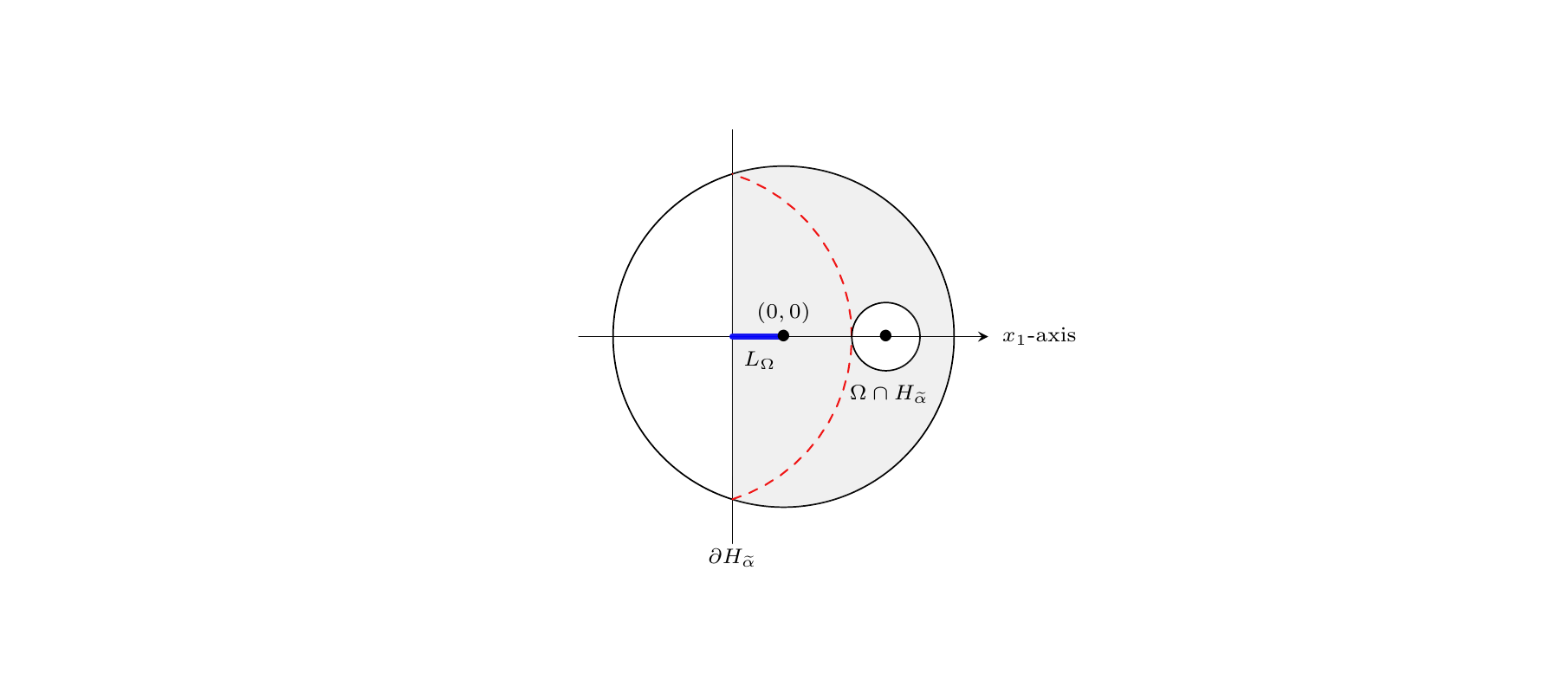}
    \caption{Location of maxima of $\phi$ in $B_R(0)\setminus\overline{B_r(te_1)}$.}
    \label{Annulus_pic}
\end{figure} 

\noi \textbf{Proof of Corollary \ref{Maxima_location}:}
Let $\Om=B_{R}(0)\setminus \overline{B_r(te_1)}$, where $0< t<R-r$. For $\alpha\in \R$, let $H_{\alpha}\in \H$ be defined as
$$H_\alpha=\left\{\big(x_1,x_2,\dots,x_N\big)\in \R^N: x_1> \alpha \right\}.$$
Then it is easy to observe that $\Om_{H_\alpha}=\Om,\;\forall\,\alpha\leq -\frac{R+r-t}{2}$ (however, if $\alpha> -\frac{R+r-t}{2}$, then $\Om_{H_\alpha}\neq\Om$). Let $\widetilde{\alpha}=-\frac{R+r-t}{2}$. Then $\Omu=\Om$ and obviously $\sigma_{\Halu}(\Om)\neq \Om$. Therefore by Theorem \ref{min_thm}-$(i)$, we have $\phi_{\Halu}=\phi$ in $\Om$. Since $g_0\in L^q(\Om)$ for $q>\frac{N}{2}$, by standard elliptic regularity (Proposition \ref{regularity}-$(a)$), $\phi\in C^1(\Om)$. Thus
\begin{equation}\label{mono_1}
    \phi(x)\leq \phi(\sigma_{\Halu}(x)),\;\forall\,x\in \Om\cap  \Halu^c.
\end{equation}
We recall that $L_{\Om}=\big\{x\in \Om\cap (-\R^+ e_1): x_1\geq \widetilde{\alpha}=-\frac{R+r-t}{2}\big\}$ (see Figure \ref{Annulus_pic}). Then $L_{\Om}\subset \Om\cap \overline{\Halu}$ and hence from \eqref{mono_1}, we have \begin{equation}\label{mono_2}
    \phi(x)\geq \phi(\sigma_{\Halu}(x)), \;\forall\,x\in L_\Om.
\end{equation} 
Also by $(i)$ of Remark \ref{rema_foliation}, 
\begin{equation}\label{mono_3}
    \max\limits_{x\in \Om}\phi(x)=\max\limits_{x\in \Om\cap (-\R^+ e_1)} \phi(x).
\end{equation}
Thus using \eqref{mono_2} and \eqref{mono_3}, we conclude that 
$\max\limits_{x\in \Om}\phi(x)=\max\limits_{x\in L_{\Om}}\phi(x).$ If $\ug$ is continuous, we can repeat the process, and hence the assertion follows.
\qed
\begin{remark}
We emphasize that for $g_0\in L^\infty(\Om)$ and $\frac{2N+2}{N+2}<p<\infty$,  using a stronger version of comparison principle \cite[Theorem 1.3]{Sciunzi} and adapting similar techniques as given in this article, one can prove all the symmetry results obtained in Theorem \ref{min_thm}-\ref{Foliation_minimization} and Corollary \ref{Steiner_theo}-\ref{Maxima_location}. However, when $p \neq 2$ and $g_0$ is not bounded, the extension of the results obtained in Section \ref{symmetry_section} seems challenging due to the lack of comparison principles which plays an important role in our proofs. 
\end{remark}

\section{Appendix}\label{appendix_section}
In this section, we study the existence and some properties of $\La(g, V)$. Let $X$ be as given in \eqref{A}. For $g, V \in X$,  we consider the following functionals on $\wp$: 
\begin{align*}
   G(\phi) = \dis g|\phi|^p; \;  J(\phi) = \dis |\Gr \phi|^p + V |\phi|^p, \quad \forall \, \phi \in \wp.
\end{align*}
One can verify that $G,J \in C^1(\wp, \R)$. 
\begin{remark}\label{G compact}
For $N>p$ and $g \in L^{\frac{N}{p}}(\Om)$, using \cite[Lemma 4.1]{Anoop-p} the map $G$ is compact on $\wp$. For $N \le p$, the compactness of $G$ holds from the compact embeddings of $\wp \hookrightarrow L^r(\Om)$ with $r \in (1, \infty)$ (when $N=p$) and $\wp \hookrightarrow L^{\infty}(\Om)$ (when $N < p$).
\end{remark}

The functional $J$ may not be coercive on $\wp$ for any sign-changing $V \in X$. However, in the following lemma under a suitable integrability assumption on $V^-$ we show that $J$ is coercive on $\wp$.
\begin{lemma}\label{coercive}
Let $V$ satisfies assumptions as given in \eqref{A}. Then there exists $\de_0 \in (0,1)$ such that 
\begin{align*}
    \dis  |\Gr \phi|^p + V |\phi|^p  \ge \de_0 \dis |\Gr \phi|^p, \quad \forall \phi \, \in \wp.
\end{align*}
\end{lemma}
\begin{proof}
Let $N>p$. For $\phi \in \wp$, using the embedding $\wp \hookrightarrow L^{p^*}(\Om)$ 
we get $$\int_{\Om} V^- \phi^p \le  \norm{V^-}_{\frac{N}{p}} \norm{\phi^p}_{\frac{p^*}{p}} = \norm{V^-}_{\frac{N}{p}} \norm{\phi}^p_{p^*} \le S^p  \norm{V^-}_{\frac{N}{p}} \int_{\Om} |\Gr \phi|^p.$$ Hence 
\begin{align*}
    \dis |\Gr \phi|^p + V \phi^p  \ge \dis |\Gr \phi|^p - V^-\phi^p  \ge  \left( 1-S^p \norm{V^-}_{\frac{N}{p}} \right) \dis |\Gr \phi|^p \ge \delta_0 \dis |\Gr \phi|^p,
\end{align*}
$\forall \, \phi \in \wp$. Therefore, the functional $J$ is coercive on $\wp$. For $N \le p$, the coercivity of $J$ follows using same arguments. 
\end{proof}

\begin{theorem} \label{Existence1}
Let $\Om$ be a bounded domain in $\R^N$.  Assume that $g, V$ satisfies \eqref{A}. Then
$$\La(g,V) =  \inf \left\{\frac{\int_{\Om} |\Gr \phi|^p + V\phi^p } { \int_{\Om} g \phi^p}: \phi \in \wp, \int_{\Om} g \phi^p > 0  \right\}$$ is attained. Moreover, $\La(g,V)$ is principal and simple. 
\end{theorem}

\begin{proof}
Due to the homogeneity of the Rayleigh quotient, we write 
\begin{align*}
  \La =  \inf \left\{ \frac{J(\phi)}{G(\phi)}: \phi \in \wp, G(\phi)> 0\right\}  = \inf \left\{ J(\phi) : \phi \in \wp, G(\phi)=1  \right\}.
\end{align*}
\textbf{Existence of $\La(g, V)$}:  Let $(\phi_n)$ be a minimizing sequence in $\wp$ such that $J(\phi_n) \ra \La(g, V)$ as $n \ra \infty$. By Lemma \ref{coercive}, the sequence $(\phi_n)$ is bounded in $\wp$. By the reflexivity, up to a subsequence $\phi_n \rightharpoonup \Phi_1$ in $\wp$. Since $N$ is weakly closed by the compactness of $G$ (Remark \ref{G compact}), $\Phi_1 \in N$. Moreover, using the lower semicontinuity of $\norm{\Gr (\cdot)}_p$, 
\begin{align*}
    \La(g,V) = \lim_{n \ra \infty} \dis  |\Gr \phi_n|^p + V |\phi_n|^p \ge \dis |\Gr \Phi_1|^p + V |\Phi_1|^p  \ge \La(g, V). 
\end{align*}
Thus $\La(g, V)$ is attained, and $\Phi_1$ is a critical point of $J$ on $N$. Therefore, by the Lagrange multiplier, $\La(g, V)$ is an eigenvalue of \eqref{EVP} and $\Phi_1$ is an eigenfunction corresponding to $\La(g, V)$.

\noi \textbf{$\La(g, V)$ is principal}: Let $\Phi_1$ be an eigenfunction of \eqref{EVP} corresponding to $\La(g, V)$. Then $|\Phi_1| \in \wp$ is also an eigenfunction  corresponding to $\La(g, V)$. For $\psi \in \c1(\Om)$ with $\psi \ge 0$, 
\begin{align*}
    \dis |\Gr (|\Phi_1|)|^{p-2} \Gr (|\Phi_1|) \cdot \Gr \psi + (V^+ +  \La g^-) |\Phi_1|^{p-1}\psi = \dis (V^- + \La g^+)|\Phi_1|^{p-1}\psi \ge 0. 
\end{align*}
Moreover, $V^+ +  \La g^- \ge 0$ and $$\int_{\Om} |V^+ + \La g^-| |\Phi_1|^{p-1} \le \left( \int_{\Om} |V^+ + \La g^-| \right)^{\frac{1}{p}} \left( \int_{\Om} |V^+ + \La g^-| |\Phi_1|^p \right)^{\frac{1}{\p}} < \infty.$$ Thus $|\Phi_1| \in \wp$ satisfies all the properties of \cite[Proposition 3.2]{KLP} (for $N>p$) and \cite[part (b) of Corollary 3.3]{KLP} (for $N \le p$). Therefore, $|\Phi_1|>0$ a.e. in $\Om$.

\noi \textbf{$\La(g, V)$ is simple}:  Suppose $\Phi_1$ and $\Phi_2$ are two eigenfunctions of \eqref{EVP} corresponding to $\La(g,V)$. Without loss of generality we assume that $\Phi_1,\Phi_2>0$ a.e. in $\Om.$ Set $P(\Phi_1,\Phi_2) := |\Gr \Phi_1|^p + (p-1) \frac{\Phi_1^p}{\Phi_2^p} \abs{\Gr \Phi_2}^p - p \frac{\Phi_1^{p-1}}{\Phi_2^{p-1}} \abs{\Gr \Phi_2}^{p-2} \Gr \Phi_2$ and $ R(\Phi_1,\Phi_2) := \abs{\Gr \Phi_1}^p - \abs{\Gr \Phi_2}^{p-2} \Gr \left( \frac{\Phi_1^p}{\Phi_2^{p-1}} \right) \cdot  \Gr \Phi_2.$
Let $\ep > 0$ be given. Then using the Picone's identity (\cite[Theorem 1.1]{Allegretto-Huang1998}),
\begin{align*}
\dis P(\Phi_1, \Phi_2 + \ep) = \dis R(\Phi_1, \Phi_2 + \ep) & = \dis  |\Gr \Phi_1|^p - |\Gr \Phi_2|^{p-2} \Gr \left( \frac{\Phi_1^p}{(\Phi_2 + \ep)^{p-1}} \right) \cdot \Gr \Phi_2  \\
& = \dis (\La g - V) \left( \Phi_1^p - \Phi_2^{p-1} \frac{\Phi_1^p}{(\Phi_2 + \ep)^{p-1}} \right).
\end{align*}
Now we let $\ep \rightarrow 0$ and apply the dominated convergence theorem to get $\int_{\Om} P(\Phi_1,\Phi_2) = 0.$ Since $P(\Phi_1,\Phi_2) \ge 0$, we obtain $P(\Phi_1, \Phi_2) = 0$ a.e. in $\Om$. Therefore, again using the Picone's identity (\cite[Theorem 1.1]{Allegretto-Huang1998}), we get that $\Phi_1$ is a constant multiple of $\Phi_2$. Thus $\La$ is simple.
\end{proof}

\section{Acknowledgments}
 The authors are grateful to Prof. T. V. Anoop for his valuable suggestions and comments, which improved the article. The second author acknowledges the support of the Israel Science Foundation (grant 637/19) founded by the Israel Academy of Sciences and Humanities.

\begin{thebibliography}{10}

\bibitem{Allegretto-Huang1998}
W.~Allegretto and Y.~X. Huang.
\newblock A {P}icone's identity for the {$p$}-{L}aplacian and applications.
\newblock {\em Nonlinear Anal.}, 32(7):819--830, 1998.

\bibitem{Anedda}
C.~Anedda and F.~Cuccu.
\newblock Steiner symmetry in the minimization of the first eigenvalue in
  problems involving the {$p$}-{L}aplacian.
\newblock {\em Proc. Amer. Math. Soc.}, 144(8):3431--3440, 2016.

\bibitem{Anoop-p}
T.~V. Anoop.
\newblock Weighted eigenvalue problems for the {$p$}-{L}aplacian with weights
  in weak {L}ebesgue spaces.
\newblock {\em Electron. J. Differential Equations}, pages No. 64, 22, 2011.

\bibitem{Anoop2020Ashok}
T.~V. Anoop, K.~Ashok~Kumar, and S.~Kesavan.
\newblock A shape variation result via the geometry of eigenfunctions.
\newblock {\em J. Differential Equations}, 298:430--462, 2021.

\bibitem{AMM}
T.~V. Anoop, M.~Lucia, and M.~Ramaswamy.
\newblock Eigenvalue problems with weights in {L}orentz spaces.
\newblock {\em Calc. Var. Partial Differential Equations}, 36(3):355--376,
  2009.

\bibitem{Ashbaugh}
M.~S. Ashbaugh and E.~M. Harrell, II.
\newblock Maximal and minimal eigenvalues and their associated nonlinear
  equations.
\newblock {\em J. Math. Phys.}, 28(8):1770--1786, 1987.

\bibitem{Berestycki2005}
H.~Berestycki, F.~Hamel, and L.~Roques.
\newblock Analysis of the periodically fragmented environment model. {I}.
  {S}pecies persistence.
\newblock {\em J. Math. Biol.}, 51(1):75--113, 2005.

\bibitem{Bianchi2020}
G.~Bianchi, R.~J. Gardner, P.~Gronchi, and M.~Kiderlen.
\newblock Rearrangement and polarization.
\newblock {\em Adv. Math.}, 374:107380, 51, 2020.

\bibitem{Brezis-Ponce}
H.~Brezis and A.~C. Ponce.
\newblock Remarks on the strong maximum principle.
\newblock {\em Differential Integral Equations}, 16(1):1--12, 2003.

\bibitem{Brock2003}
F.~Brock.
\newblock Symmetry and monotonicity of solutions to some variational problems
  in cylinders and annuli.
\newblock {\em Electron. J. Differential Equations}, pages No. 108, 20, 2003.

\bibitem{BrockF}
F.~Brock.
\newblock Positivity and radial symmetry of solutions to some variational
  problems in {$\Bbb R^N$}.
\newblock {\em J. Math. Anal. Appl.}, 296(1):226--243, 2004.

\bibitem{Brock2020}
F.~Brock, G.~Croce, O.~Guib\'{e}, and A.~Mercaldo.
\newblock Symmetry and asymmetry of minimizers of a class of noncoercive
  functionals.
\newblock {\em Adv. Calc. Var.}, 13(1):15--32, 2020.

\bibitem{Brock2000}
F.~Brock and A.~Y. Solynin.
\newblock An approach to symmetrization via polarization.
\newblock {\em Trans. Amer. Math. Soc.}, 352(4):1759--1796, 2000.

\bibitem{Brothers}
J.~E. Brothers and W.~P. Ziemer.
\newblock Minimal rearrangements of {S}obolev functions.
\newblock {\em J. Reine Angew. Math.}, 384:153--179, 1988.

\bibitem{Burton}
G.~R. Burton.
\newblock Variational problems on classes of rearrangements and multiple
  configurations for steady vortices.
\newblock {\em Ann. Inst. H. Poincar\'{e} Anal. Non Lin\'{e}aire},
  6(4):295--319, 1989.

\bibitem{Cadeddu2011Porru}
L.~Cadeddu and G.~Porru.
\newblock Symmetry breaking in problems involving semilinear equations.
\newblock {\em Discrete Contin. Dyn. Syst.}, (Dynamical systems, differential
  equations and applications. 8th AIMS Conference. Suppl. Vol. I):219--228,
  2011.

\bibitem{Cantrell}
R.~S. Cantrell and C.~Cosner.
\newblock Diffusive logistic equations with indefinite weights: population
  models in disrupted environments.
\newblock {\em Proc. Roy. Soc. Edinburgh Sect. A}, 112(3-4):293--318, 1989.

\bibitem{Robert}
R.~S. Cantrell and C.~Cosner.
\newblock Diffusive logistic equations with indefinite weights: population
  models in disrupted environments. {II}.
\newblock {\em SIAM J. Math. Anal.}, 22(4):1043--1064, 1991.

\bibitem{Chanillo}
S.~Chanillo, D.~Grieser, M.~Imai, K.~Kurata, and I.~Ohnishi.
\newblock Symmetry breaking and other phenomena in the optimization of
  eigenvalues for composite membranes.
\newblock {\em Comm. Math. Phys.}, 214(2):315--337, 2000.

\bibitem{Cianchi}
A.~Cianchi and N.~Fusco.
\newblock Steiner symmetric extremals in {P}\'{o}lya-{S}zeg\"{o} type
  inequalities.
\newblock {\em Adv. Math.}, 203(2):673--728, 2006.

\bibitem{Cox}
S.~J. Cox and J.~R. McLaughlin.
\newblock Extremal eigenvalue problems for composite membranes. {I}, {II}.
\newblock {\em Appl. Math. Optim.}, 22(2):153--167, 169--187, 1990.

\bibitem{Cuccu}
F.~Cuccu, B.~Emamizadeh, and G.~Porru.
\newblock Optimization of the first eigenvalue in problems involving the
  {$p$}-{L}aplacian.
\newblock {\em Proc. Amer. Math. Soc.}, 137(5):1677--1687, 2009.

\bibitem{Cuccu2002Porru}
F.~Cuccu and G.~Porru.
\newblock Optimization in a problem of heat conduction.
\newblock {\em Adv. Math. Sci. Appl.}, 12(1):245--255, 2002.

\bibitem{Cuesta2009}
M.~Cuesta and H.~Ramos~Quoirin.
\newblock A weighted eigenvalue problem for the {$p$}-{L}aplacian plus a
  potential.
\newblock {\em NoDEA Nonlinear Differential Equations Appl.}, 16(4):469--491,
  2009.

\bibitem{Lucio}
L.~Damascelli and R.~Pardo.
\newblock A priori estimates for some elliptic equations involving the
  {$p$}-{L}aplacian.
\newblock {\em Nonlinear Anal. Real World Appl.}, 41:475--496, 2018.

\bibitem{Leandro}
L.~M. Del~Pezzo and J.~Fern\'{a}ndez~Bonder.
\newblock An optimization problem for the first weighted eigenvalue problem
  plus a potential.
\newblock {\em Proc. Amer. Math. Soc.}, 138(10):3551--3567, 2010.

\bibitem{Takac2010}
A.~Derlet, J.-P. Gossez, and P.~Tak\'{a}\v{c}.
\newblock Minimization of eigenvalues for a quasilinear elliptic {N}eumann
  problem with indefinite weight.
\newblock {\em J. Math. Anal. Appl.}, 371(1):69--79, 2010.

\bibitem{Prajapat}
B.~Emamizadeh and J.~V. Prajapat.
\newblock Symmetry in rearrangement optimization problems.
\newblock {\em Electron. J. Differential Equations}, pages No. 149, 10, 2009.

\bibitem{Bonder}
J.~Fern\'{a}ndez~Bonder and L.~M. Del~Pezzo.
\newblock An optimization problem for the first eigenvalue of the
  {$p$}-{L}aplacian plus a potential.
\newblock {\em Commun. Pure Appl. Anal.}, 5(4):675--690, 2006.

\bibitem{Guedda}
M.~Guedda and L.~V\'{e}ron.
\newblock Quasilinear elliptic equations involving critical {S}obolev
  exponents.
\newblock {\em Nonlinear Anal.}, 13(8):879--902, 1989.

\bibitem{Henrot2006}
A.~Henrot.
\newblock {\em Extremum problems for eigenvalues of elliptic operators}.
\newblock Frontiers in Mathematics. Birkh\"{a}user Verlag, Basel, 2006.

\bibitem{Porru2011Jha}
K.~Jha and G.~Porru.
\newblock Minimization of the principal eigenvalue under {N}eumann boundary
  conditions.
\newblock {\em Numer. Funct. Anal. Optim.}, 32(11):1146--1165, 2011.

\bibitem{Kawohl1985}
B.~Kawohl.
\newblock {\em Rearrangements and convexity of level sets in {PDE}}, volume
  1150 of {\em Lecture Notes in Mathematics}.
\newblock Springer-Verlag, Berlin, 1985.

\bibitem{KLP}
B.~Kawohl, M.~Lucia, and S.~Prashanth.
\newblock Simplicity of the principal eigenvalue for indefinite quasilinear
  problems.
\newblock {\em Adv. Differential Equations}, 12(4):407--434, 2007.

\bibitem{Krein}
M.~G. Krein.
\newblock On certain problems on the maximum and minimum of characteristic
  values and on the {L}yapunov zones of stability.
\newblock {\em Amer. Math. Soc. Transl. (2)}, 1:163--187, 1955.

\bibitem{Kurata2004}
K.~Kurata, M.~Shibata, and S.~Sakamoto.
\newblock Symmetry-breaking phenomena in an optimization problem for some
  nonlinear elliptic equation.
\newblock {\em Appl. Math. Optim.}, 50(3):259--278, 2004.

\bibitem{Lamboley}
J.~Lamboley, A.~Laurain, G.~Nadin, and Y.~Privat.
\newblock Properties of optimizers of the principal eigenvalue with indefinite
  weight and {R}obin conditions.
\newblock {\em Calc. Var. Partial Differential Equations}, 55(6):Art. 144, 37,
  2016.

\bibitem{Leadi}
L.~Leadi and A.~Yechoui.
\newblock Principal eigenvalue in an unbounded domain with indefinite
  potential.
\newblock {\em NoDEA Nonlinear Differential Equations Appl.}, 17(4):391--409,
  2010.

\bibitem{Lieberman1988}
G.~M. Lieberman.
\newblock Boundary regularity for solutions of degenerate elliptic equations.
\newblock {\em Nonlinear Anal.}, 12(11):1203--1219, 1988.

\bibitem{Mazari2020}
I.~Mazari, G.~Nadin, and Y.~Privat.
\newblock Optimal location of resources maximizing the total population size in
  logistic models.
\newblock {\em J. Math. Pures Appl. (9)}, 134:1--35, 2020.

\bibitem{Mazzolini}
D.~Mazzoleni, B.~Pellacci, and G.~Verzini.
\newblock Asymptotic spherical shapes in some spectral optimization problems.
\newblock {\em J. Math. Pures Appl. (9)}, 135:256--283, 2020.

\bibitem{Pielichowski}
W.~a. Pielichowski.
\newblock The optimization of eigenvalue problems involving the
  {$p$}-{L}aplacian.
\newblock {\em Univ. Iagel. Acta Math.}, (42):109--122, 2004.

\bibitem{Sciunzi}
B.~Sciunzi.
\newblock Regularity and comparison principles for {$p$}-{L}aplace equations
  with vanishing source term.
\newblock {\em Commun. Contemp. Math.}, 16(6):1450013, 20, 2014.

\bibitem{Skellam}
J.~G. Skellam.
\newblock Random dispersal in theoretical populations.
\newblock {\em Biometrika}, 38:196--218, 1951.

\bibitem{Szulkin}
A.~Szulkin and M.~Willem.
\newblock Eigenvalue problems with indefinite weight.
\newblock {\em Studia Math.}, 135(2):191--201, 1999.

\bibitem{Scaftingen2005}
J.~Van~Schaftingen.
\newblock Symmetrization and minimax principles.
\newblock {\em Commun. Contemp. Math.}, 7(4):463--481, 2005.

\bibitem{Weth2010}
T.~Weth.
\newblock Symmetry of solutions to variational problems for nonlinear elliptic
  equations via reflection methods.
\newblock {\em Jahresber. Dtsch. Math.-Ver.}, 112(3):119--158, 2010.

\end{thebibliography}

\end{document}